\documentclass[10pt]{amsart}
\usepackage{fullpage}
\usepackage{amsfonts, amssymb, amsthm, amsmath, amscd, euscript}
\usepackage[all]{xypic}
\usepackage{comment}

\usepackage{tikz-cd}
\usepackage{textcomp}
\usepackage{rotating}
\usepackage{mathbbol}
\usepackage{xcolor}

\numberwithin{equation}{section}
\theoremstyle{plain}
\newtheorem{theorem}[equation]{Theorem}
\newtheorem{proposition}[equation]{Proposition}
\newtheorem{corollary}[equation]{Corollary}
\newtheorem{lemma}[equation]{Lemma}

\theoremstyle{definition}
\newtheorem{definition}[equation]{Definition}
\newtheorem{defn}[equation]{Definition}

\newtheorem{notation}[equation]{Notation}

\newtheorem{example}[equation]{Example}

\newtheorem{remark}[equation]{Remark}

\newcommand{\beq}{\begin{equation}}
\newcommand{\eeq}{\end{equation}}

\newcommand{\mb}{\mathbb}
\newcommand{\mc}{\mathcal}
\newcommand{\mf}{\mathfrak}

\newcommand{\kk}{{\Bbbk}}

\newcommand{\CC}{{\Bbbk}}

\newcommand{\NN}{\mb Z_{\ge0}} 

\newcommand{\ZZ}{\mb{Z}}

\newcommand{\DMO}{\DeclareMathOperator}
\DMO{\GK}{GKdim}
\DMO{\PGK}{PGKdim}
\DMO{\gr}{gr}
\DMO{\Der}{Der}
\DMO{\trdeg}{trdeg}
\DMO{\Kdim}{Kdim}
\DMO{\codim}{codim}
\DMO{\ord}{ord}
\DMO{\FPol}{Det}

\DMO{\Res}{Res}
\DMO{\MSpec}{MSpec}
\DMO{\Sa}{S}
\DMO{\Ua}{U}
\DMO{\Core}{Core}
\DMO{\rk}{rk}
\DMO{\Ker}{\ker}
\DMO{\End}{End}
\DMO{\Aut}{Aut}
\DMO{\im}{im}
\DMO{\Ann}{Ann}
\DMO{\supp}{supp}
\DMO{\ldeg}{ldeg}
\DMO{\sdeg}{gaps}
\DMO{\CoreP}{P}
\DMO{\diag}{diag}
\DMO{\Sat}{Sat}
\DMO{\Sym}{Sym}
\DMO{\GKdim}{GKdim}

\DMO{\LslT}{\mathrm{L}\mf{sl}_2}
\DMO{\Lg}{\mathrm{L} \mf {g}}

\DMO{\LslTH}{\mathrm{L}\mf{sl}_3}
\DMO{\GslT}{\mathrm{L}\mathrm{SL}(2)}
\DMO{\Sg}{\mf S}
\newcommand{\Lh}{{\mathrm L} h}
\DMO{\Eu}{Eu}
\DMO{\lie}{lie}
\DMO{\PSpec}{PSpec}
\DMO{\len}{len}
\DMO{\Spec}{Spec}
\DMO{\PhiP}{\underline{\Phi}}

\newcommand{\Vir}{V\!ir}

\newcommand{\del}{\partial}
\newcommand{\ssm}{\setminus}

\newcommand{\Lmf}{\mathrm{L}\mathfrak}

\DMO{\ev}{ev}
\DMO{\lyr}{filt}
\DMO{\Ryn}{Reyn}

\newcommand{\U}{{\rm U}}

\newcommand{\normal}{\vartriangleleft}
\providecommand{\abs}[1]{|#1|}

\DMO{\res}{res}

\renewcommand{\baselinestretch}{1.49}

\newcommand{\wh}{\widehat}
\DMO{\I}{I}

\DMO{\Stab}{Stab}

\begin{document}

\title{Ideals in Enveloping Algebras of Loop Algebras}
\author{Alexey V. Petukhov and Susan J. Sierra}
\date{\today}

\address{Sierra: School of Mathematics and Maxwell Institute, University of Edinburgh, Edinburgh EH9 3FD, United Kingdom}

\email{s.sierra@ed.ac.uk}

\address{Petukhov: Institute for Information Transmission Problems, Bolshoy Karetniy 19-1, Moscow 127994, Russia}
\email{alex-{}-2@yandex.ru}

\renewcommand{\baselinestretch}{1.49}

\begin{abstract}
 Let $\mf g$ be a finite-dimensional simple complex Lie algebra, and consider the loop algebra $\Lg = \mf g[s, s^{-1}]$
 and the affine Lie algebra $\wh{\mf g}$, which is an extension of $\Lg$ by a central element $c$. 
 It is known that the rings $\Ua(\wh{\mf g})/(c-\lambda)$ 
 are simple for any nonzero scalar $\lambda$.
 On the other hand, the universal enveloping algebra 
 $\Ua(\Lg) \cong \Ua(\wh{\mf g})/(c)$ has many two-sided ideals,
 and 
 in this paper we 
 investigate their structure.  
 We show  that the two-sided structure of $\Ua(\Lg)$ is both intricate and well-behaved.
Specifically, $\Ua(\Lg)$ does not satisfy the ascending chain condition on two-sided ideals, but  the two-sided ideals still have a nice structure: 
there is a canonical nested collection of ideals $I_n$, parameterised by positive integers, so that any nonzero two-sided ideal of $\Ua(\Lg)$ contains all $I_n$ for $n \gg 0$. The ideals $I_n$ can be thought of as universal annihilators of natural classes of finite-dimensional representations of $\Lg$. 
We also study the Poisson structure of the symmetric algebra $\Sa(\Lg)$ and give applications to the representation theory of $\Lg$.

\end{abstract}


\maketitle

\tableofcontents

\section{Introduction}\label{INTRO2}

In the introduction, let  the ground field be  $\kk = \mb C$; in the body of the paper we will allow $\kk$ to be
an arbitrary algebraically closed field of characteristic 0.
Let $\mf g$ be a finite-dimensional simple Lie algebra and consider the {\em loop algebra} $\Lg = \mf g \otimes \kk[s,s^{-1}]$ and its universal enveloping algebra $\Ua(\Lg)$.
It is well known that $\Ua(\Lg)$ is neither left nor right noetherian,  but the two-sided structure of $\Ua(\Lg)$ is, a priori, much less clear.
An illustration of the subtleties is that if $\widehat{\mf g}$ is the affine Lie algebra associated to $\mf g$, so $\wh{\mf g}$ has a central element $c$ with $\wh{\mf g}/(c) \cong \Lg$,  for any $\lambda \in \kk$ then $\Ua(\wh{\mf g})/(c-\lambda)$ fails to be left or right noetherian (see the beginning of subsection~\ref{ss:nonnoeth}).
On the other hand, by \cite[Theorem~1.1(1)]{BS} if $\lambda \neq 0$ then $\Ua(\wh{\mf g})/(c-\lambda)$ is a simple ring.

This paper is devoted to the two-sided structure of $\Ua(\Lg) \cong \Ua(\wh{\mf g})/(c)$.
In contrast to nontrivial central quotients of $\Ua(\wh{\mf g})$, we will see that $\Ua(\Lg)$ has many two-sided ideals, and in fact (Theorem~\ref{thm:nonnoeth}) does not satisfy the ascending chain condition on two-sided ideals.
Paradoxically, however, our main result is that two-sided ideals of $\Ua(\Lg)$ also have a great deal of structure.

This structure is intimately related to the representation theory of $\Lg$.
Recall that if $V$ is a representation of $\mf g$ and $a \in \kk^\times$, there is a natural {\em evaluation representation} $V_a$ of $\Lg$, where $gs^i \in \Lg$ acts on $V$ by $a^ig$.
By construction, the action of $\Lg$ on $V_a$ factors through the {\em universal evaluation homomorphism}
\begin{align*} \Phi_1: \Ua(\Lg) & \to \Ua(\mf g) \otimes \kk[t,t^{-1}] \\
    g s^i & \mapsto g t^i
\end{align*}
Define $\Phi_n$ to be the composition
\[\xymatrix{
\Ua(\Lg) \ar@/_1pc/[rr]_{\Phi_n} \ar[r]^{\Delta^n} & \Ua(\Lg)^{\otimes n} \ar[r]^{(\Phi_1)^{\otimes n}} & (\Ua(\mf g)[t, t^{-1}])^{\otimes n}, 
}\]
where $\Delta$ is the coproduct on $\Ua(\Lg)$.
Our main result is that the kernels of the $\Phi_n$ are cofinal in the ideal lattice of $\Ua(\Lg)$.
That is:
\begin{theorem}[Theorem~\ref{TmainI+}]\label{ithm:main}
Let $K$ be a nonzero ideal of $\Ua(\Lg)$.
Then there exists $n \geq 1$ so that $K \supseteq \ker \Phi_n$.
Further, since $\ker \Phi_n \supseteq \ker \Phi_{n+1}$, thus $K \supseteq \ker \Phi_m$ for all $m \geq n$.
\end{theorem}
This result provides a precise version for $\Ua(\Lg)$ of the general countable separation property for associative algebras, see~\cite{Irving1979, Petukhov2025}.

Besides giving insight into the two-sided structure of $\Ua(\Lg)$, Theorem~\ref{ithm:main} has applications to the representation theory of $\Lg$.
As an immediate corollary, we see that if $W$ is a finite-dimensional representation of $\Lg$, then there exists $n \geq 1$ so that $(\ker \Phi_n) \cdot W = 0$.
Note that if $V^1, \dots, V^n$ are representations of $\mf g$ and $a_1, \dots, a_n \in \kk^\times$, then by construction
$(\ker \Phi_n)\cdot (V^1_{a_1} \otimes \cdots \otimes V^n_{a_n}) = 0$, but the statement for general $W$ does not seem to be previously known.

Furthermore, we are able to explicitly describe the ideals $\ker \Phi_n$ (see Lemma~\ref{lem:Phi evaluation} and Proposition~\ref{prop:kernel}) and as a result show (Corollary~\ref{cor:kernew})
that $0$ is the only element of $\Ua(\Lg)$ that annihilates all finite-dimensional irreducible representations of $\Ua(\Lg)$. 
By contrast, for $\Ua(\wh{\mf g})$ we have
\[(c) =  \bigcap \{ \Ann_{\Ua(\wh{\mf g})}(W) \mid \mbox{ $W$ is a finite-dimensional irreducible representation of $\wh{\mf g}$ } \}, \]
since all finite-dimensional irreducible representations of $\wh{\mf g}$ have trivial central character. 

Our main method for proving Theorem~\ref{ithm:main} is to intersect an  ideal  with a large commutative subalgebra of $\Ua(\Lg)$, and study the structure of the intersection, which is easier to work with as an ideal of a commutative ring.
Let $\mf b \subseteq \mf g$ be a Borel subalgebra and let $\mf h \subseteq \mf b$ be the associated Cartan subalgebra.
Let $e \in \mf b$ be a highest weight vector for the adjoint action of $\mf h$ and let $E = \mathrm{L} e = e \otimes \kk[s,s^{-1}] \subseteq \Lg$.
Since $E$ is abelian, the subalgebra $\Ua(E)$ of $\Ua(\Lg)$ is isomorphic to the symmetric algebra $\Sa(E)$,
and more explicitly 
\[ \Ua(E) \cong \Sa(E) \cong \kk[e_i \mid i \in \ZZ],\]
where $e_i = es^i \in E$.
Now let $ K  \neq (0)$ be an ideal of $\Ua(\Lg)$.
By \cite[Theorem~1.1(0)]{BS}, $\Ua(\Lg)/K$ has polynomial growth, that is finite GK-dimension,
and it follows that $K \cap \Sa(E) \neq (0)$.
Furthermore, $H := \mathrm{L} \mf h$ acts on $\Sa(E)$ via the adjoint action, and $K\cap \Sa(E)$ is closed under this action: we say that $K \cap \Sa(E)$ is an {\em $H$-ideal} of $\Sa(E)$.
To understand ideals of $\Ua(\Lg)$ we are thus naturally led to study the structure of (nontrivial) $H$-ideals of $\Sa(E)$.
(We note here that this structure is independent of the specific Lie algebra $\mf g$ with which we began.)

As an $H$-ideal of $\Sa(E)$ is also an $\mf h \cong \mf h \otimes 1$-subrepresentation, $H$-ideals of $\Sa(E)$ are automatically graded.
Furthermore, if we let $\mc Z_k := \kk[z_1^{\pm 1}, \dots, z_k^{\pm 1}]$, with the natural permutation action of the symmetric group $\Sg_k$,
then an easy computation (Lemma~\ref{Lthid}) gives that the bijective maps
\begin{align*} \Theta^k_\star: \Sa^k(E) &  \longleftrightarrow \mc Z_k^{\Sg_k}: \Theta_k^\star \\
 e_{i_1}e_{i_2} \dots e_{i_k} & \mapsto \Ryn_k(z_1^{i_1}z_2^{i_2} \dots z_k^{i_k}) = \frac{1}{k!} \sum_{\sigma \in \Sg_k} \sigma(z_1^{i_1}z_2^{i_2} \dots z_k^{i_k})
 \end{align*}
 intertwine $H$-subrepresentations of $\Sa^k(E)$ with ideals of $Z_k^{\Sg_k}$.
 (We caution the reader that a slightly more complicated notation is used for these maps in the body of the paper; see \eqref{Theta}.)
 This allows us to relate $H$-ideals of $\Sa(E)$ to ideals of $\mc Z_\infty = \bigcup_k \mc Z_k = \kk[z_1^{\pm 1}, z_2^{\pm 1}, \dots]$
which are invariant under the permutation action of the infinite symmetric group $\Sg_\infty$.
That is, there are {\em $\Theta$-maps}
\[ \Theta_\star: \{\mbox{ $H$-ideals of $\Sa(E)$ } \} \leftrightarrows \{ \mbox{ $\Sg_\infty$-invariant ideals of $\mc Z_\infty$ } \}: \Theta^\star\]
(see Definition~\ref{Dtheta} for the precise definitions).  

The $\Theta$-maps are not bijections, as the examples in Section~\ref{EXAMPLES} show,
but they are still extremely useful to our project, since $\Sg_\infty$-invariant ideals of $\mc Z_\infty$ have a great deal of structure.  
The main facts we use are summarised in:
\begin{theorem}\label{ithm:A}
\begin{enumerate}
    \item (Cohen) The lattice of $\Sg_\infty$-invariant ideals of $\mc Z_\infty$ satisfies the ascending chain condition.
    \item (Nagpal-Snowden) If $ J_\star \neq (0)$ is an $\Sg_\infty$-invariant ideal of $\mc Z_\infty$, then there exists $n \geq 1$ so that
  \[ \FPol_n := \prod_{1 \leq i< j \leq n} (z_i -z_j) \in J_\star.\]
\end{enumerate}
\end{theorem}
 The paper~\cite{NS} by Nagpal and Snowden is available only as a preprint; we therefore provide a self-contained proof, following their argument,  of Theorem~\ref{ithm:A}(2) in Proposition~\ref{P:NS}.

By using the $\Theta$-maps, we are able to transfer this structure to $\Sa(E)$ and prove:
\begin{theorem}\label{ithm:B}
\begin{enumerate}
    \item {\rm (Corollary~\ref{Cacc}).} The lattice of $H$-ideals of $\Sa(E)$ satisfies the ascending chain condition.
\item {\rm (Theorem~\ref{Disc for H ideals}).} If $ I^\star \neq (0)$ is an $H$-ideal of $\Sa(E)$, then there exists $n \geq 1$ so that
\[ D_n := \Theta^\star_n(\FPol_n^2) \in I^\star.\]
\end{enumerate}
\end{theorem}
(In the statement of $(2)$, note that $\FPol_n^2 \in \mc Z_n^{\Sg_n}$, whereas $\FPol_n$ is not.)

In spite of the similarity of the statements of Theorems~\ref{ithm:A} and \ref{ithm:B}, the proof of Theorem~\ref{ithm:B} is delicate, and takes all of Sections~\ref{H IDEALS} and \ref{S:disc}.
Note also the contrast between Theorem~\ref{ithm:B}(1) and the non-noetherianity of the ideal lattice of $\Ua(\Lg)$.
That is, the passage from ideals of $\Ua(\Lg)$ to $H$-ideals of $\Sa(E)$ is one-way, and $H$-ideals do not capture all information about ideals of $\Ua(\Lg)$.

Nevertheless, our main result Theorem~\ref{ithm:main} follows relatively quickly from Theorem~\ref{ithm:B}(2).
Let $ K \neq (0)$ be an ideal of $\Ua(\Lg)$.
Since $K \cap \Sa(E) \neq (0)$, by Theorem~\ref{ithm:B}(2) there is $n \geq 1$ so that $D_n \in K\cap \Sa(E)$.
The description of  the $\ker \Phi_k$ in Proposition~\ref{prop:kernel} now allows us to conclude
(Proposition~\ref{Pkto3k}) that $K \supseteq \ker \Phi_{d(n-1)}$, where $d = \dim \mf g$, giving the result.

From the above discussion,   two-sided ideals of $\Ua(\Lg)$ are closely related to  $H$-ideals of $\Sa(E)$. 
The latter ideals can be studied by tools of differential algebra. 
Following these lines we introduce the concept of $H$-primitive ideals.
Every radical $H$-ideal is an intersection of a family of $H$-primitive ideals --- in particular these ideals are the points of the topological space defined by the lattice of radical $H$-ideals. 
This construction mimics the construction of the Jacobson (Zariski) topology for associative algebras. 
We show that $H$-primitive ideals are exactly the $H$-ideals generated by their intersection with $E\subseteq\Sa(E)$.

There are also Poisson versions of many of our results.
First, it follows by standard associated graded techniques
that, since the ideal lattice of $\Ua(\Lg)$ does not satisfy the ascending chain condition, neither does the Poisson ideal lattice of $\Sa(\Lg)$.
However, this is proved directly in Theorem~\ref{thm:nonnoeth}.
Further,  \cite{BS} also shows that if $K$ is a nonzero Poisson ideal of the symmetric algebra $\Sa(\Lg)$, then $\Sa(\Lg)/K$ has polynomial growth.
Thus,  $K \cap \Sa(E)$ is again a nontrivial $H$-ideal, and we obtain a Poisson version of our main result Theorem~\ref{ithm:main}.
\begin{theorem}[Theorem~\ref{TPmainI+}]\label{ithm:Pmain}
For all $n \geq 1$, let
\[ \underline{\Phi}_n: \Sa(\Lg) \to (\Sa(\mf g) [t, t^{-1}])^{\otimes n}
\]
be the associated graded map of $\Phi_n$.
If $ K \neq (0)$ is a Poisson ideal of $\Sa(\Lg)$, then there is $n$ so that $K \supseteq \ker {\underline{\Phi}_n}$.
\end{theorem}

We  study Poisson primitive ideals of $\Sa(\Lg)$ and show that they have a surprisingly similar characterisation (see Proposition~\ref{Ppoipr}) to that of Poisson primitive ideals in the symmetric algebra of the {\em Witt algebra} $W = \Der \kk[t,t^{-1}]$, as described in \cite{PS2}. 
We also show that although the prime Poisson ideals of $\Sa(\Lg)$ do not  satisfy the {\em Poisson  Dixmier-Moeglin equivalence (PDME)}, 
which, for symmetric algebras of  finite-dimensional Lie algebras, provides two equivalent conditions (rationality, local closedness) for a prime Poisson ideal to be Poisson primitive, 
nevertheless the PDME almost holds for $\Sa(\Lg)$, again similar to $\Sa(W)$.
Finally, we reduce the description of the Poisson primitive ideals of $\Sa(\Lg)$ to the description of primitive ideals of the symmetric algebras of an infinite family of finite-dimensional Lie algebras, and to coadjoint orbits of the loop group associated to $\Lg$.

The starting point of this project was the prior work of the authors on the symmetric and enveloping algebras of the Witt algebra and its (nontrivial) central extension $\Vir$, the {\em Virasoro algebra}.
For $W$ and $\Vir$, one can almost completely describe Poisson ideals of $\Sa(\Vir)$ and $\Sa(W)$, but the analogous statements for $\Ua(\Vir)$ and $\Ua(W)$ remain conjectural as of this writing (see Section~7 of \cite{PS2}).
However, because of the existence of the large commutative subalgebra $\Sa(E)$ of $\Ua(\Lg)$, in contrast to the situation with $\Ua(W)$ and $\Ua(\Vir)$ we have been able to work out a great deal of the two-sided structure of $\Ua(\Lg)$.

It is natural to think of loop algebras, in particular $\LslT$, as ``toy models'' for the Witt algebra and to think of affine algebras such as $\wh{\mf{sl}_2}$ as ``toy models'' for the Virasoro algebra.
The similarity between the description of Poisson primitive ideals of $\Sa(\Lg)$ in Proposition~\ref{Ppoipr} and the results of \cite{PS2} supports the validity of this approach.
It is an extremely interesting question which of the structural (non-)results in this paper for $\Ua(\Lg)$ and $\Ua(\wh{\mf g})$
have analogues which hold for $\Ua(W)$ and $\Ua(\Vir)$.

We now describe the structure of the paper.
After giving brief background in Section~\ref{BACKGROUND}, we define the maps $\Phi_n$ and prove non-noetherianity of the ideal lattice of $\Ua(\Lg)$ in Section~\ref{NONNOETHERIAN}.
(The relationship between the two is that we actually show that the ideal lattice of $\im \Phi_1 \cong \Ua(\Lg)/\ker \Phi_1$ is non-noetherian.)
In Section~\ref{H IDEALS} we introduce $H$-ideals and the $\Theta$-maps and prove fundamental results about them,
and in Section~\ref{S:disc} we prove Theorem~\ref{ithm:B}.
Then in Section~\ref{S:strid} we describe $\ker \Phi_n$ and prove Theorem~\ref{ithm:main}.
We give applications to the representation theory of $\Lg$ in Section~\ref{REPRESENTATIONS}.
In Section~\ref{POISSON} we consider what can be deduced about Poisson ideals of $\Sa(\Lg)$ from their relationship with $H$-ideals, and we end with examples in Section~\ref{EXAMPLES}.

\vspace{1em}
\noindent{\bf Use of AI: } We  used the AI tools DeepSeek and Aristotle  to flag  typos and awkward wording in the text, and 
  used Aristotle  to check proofs. Aristotle  found a few minor gaps, which we  fixed on our own.

\vspace{1em}
\noindent{\bf Acknowledgements: }  We thank Jason Bell, Vyjayanthi Chari,  Duncan Laurie, Omar Le\'on S\'anchez, and Toby Stafford  for informative and useful discussions.  The second author was partially supported by EPSRC grant EP/T018844/1 and thanks the EPSRC for their support. 

\section{Background}\label{BACKGROUND}

We establish notation, which will hold throughout the paper.
Fix an algebraically closed ground field $\kk$ of characteristic 0.
(In Section~\ref{POISSON} we will  also assume $\kk$ to be uncountable.)
All (associative, Lie) algebras are $\kk$-algebras.  
We will use $\otimes$ to denote $\otimes_\kk$.
If $V,W$ are vector spaces and $v\in V$, we will abuse notation and write $v\otimes W$ for $\kk v \otimes W$. 
We will sometimes write $I \normal R$ or $\mf k \normal \mf g$ to denote a (Lie) ideal of a ring $R$ or a Lie algebra $\mf g$.

If $\mf a$ is any Lie algebra let $\Ua(\mf a)$ be the universal enveloping algebra of $\mf a$ and let $\Sa(\mf a)$ be the symmetric algebra of $\mf a$.  Let $\kk = \Ua_{\leq 0}(\mf a) \subseteq \Ua_{\leq 1}(\mf a) \subseteq \dots$ be the canonical filtration of $\Ua(\mf a)$.
The associated graded ring of this filtration is isomorphic to $\Sa(\mf a) = \bigoplus_{k \geq 0} \Sa^k(\mf a)$ by the Poincar\'{e}-Birkhoff-Witt theorem.
If $r \in \Ua_{\leq k}(\mf a) \ssm \Ua_{\leq k-1}(\mf a)$, we write 
\[ \gr(r) = r + U_{\leq k-1}(\mf a) \in \Sa^k(\mf a).\]

As the associated graded ring of $\Ua(\mf a)$, the algebra $\Sa(\mf a)$ has a canonical induced Poisson algebra structure.
That is, there is a Lie bracket $\{ -,-\}$ on $\Sa(\mf a)$ that is a derivation in each input and such that $\{x,y\} = [x,y]$ for all $x,y \in \mf a$.
Recall that if $A$ is a Poisson algebra, then an ideal $I$ of $A$ is a {\em Poisson ideal} if $I$ is also  a Lie ideal of $A$ under the Poisson bracket.

Throughout the paper, let $\mf g $ be a finite-dimensional simple Lie algebra, of dimension $d = \dim \mf g$. 
Let $\Lg = \mf g[s, s^{-1}]$ be the loop algebra of $\mf g$.
We will occasionally use the affine Lie algebra $\wh{\mf g}$ associated to $\mf g$, that is, the unique (up to isomorphism) nontrivial central extension of $\Lg$ by the one-dimensional central subalgebra $\kk c$.

The universal enveloping algebra $\Ua(\Lg)$  of $\Lg$ and  the symmetric algebra  $\Sa(\Lg)$ of $\Lg$ are the main objects of study in this paper.
The foundational result in the structure theory of two-sided ideals of $\Ua(\Lg)$ is that these ideals are ``large''.  Similar results hold for Poisson ideals of $\Sa(\Lg)$.  
To explain this result, we recall that   the Gelfand-Kirillov (GK-) dimension of an algebra $A$ is defined to  be
\[ \sup_V \varlimsup_{n \to \infty} \log_n \dim_\kk V^n,\]
where the supremum is taken over all finite-dimensional subspaces $V \subseteq A$.
Our  reference for basic facts about GK-dimension is \cite{KL}.
We say that $A$ has {\em polynomial growth} (or {\em finite growth}) if $\GKdim A < \infty$.
Further,  $A$ has {\em just-infinite growth}  if $\GKdim A = \infty$, but any proper quotient of $A$ has finite growth: $A/J$ has finite growth for every    nonzero ideal $J$ of $A$.

\begin{theorem}\label{thm:BS}
Let $\mf g$ be a finite-dimensional simple Lie algebra and let $\Lg$ be the loop algebra of $\mf g$.
\begin{enumerate}
    \item \cite[Theorem~1.1]{BS} $\Ua(\Lg)$ has just-infinite growth.
    \item \cite[Theorem~1.2]{BS} $\Sa(\Lg)$ has just-infinite growth as a Poisson algebra:  that is, if $J$ is a nonzero Poisson ideal of $\Sa(\Lg)$, then $\Sa(\Lg)/J$ has polynomial growth.
\end{enumerate}
\end{theorem}

We will also need notation coming from group actions. 
If $G$ is a finite group acting (linearly) on a vector space $V$,
let $\Ryn_G: V \to V^G$ be the Reynolds operator:
\[ \Ryn_G(v) = \frac{1}{\abs{G}} \sum_{\sigma\in G} \sigma\cdot v.\]
Since $\Ryn_G$ fixes $G$-invariant elements, $\Ryn_G^2 = \Ryn_G$; more generally, if $H \leq G$ then
\beq\label{proj}
\Ryn_G = \Ryn_G\circ \Ryn_H.
\eeq
Note also that if $A$ is a $\kk$-algebra and $G$ acts on $A$ via algebra homomorphisms, then $\Ryn_G:A \to A$ is $A^G$-linear.
If $G = \Sg_k$ is the $k$th symmetric group, we will denote $\Ryn_{\Sg_k}$ by $\Ryn_k$. 

Now let $V$ be a vector space and let $k \geq 1$. 
We write $$\Sa^k(V) := V^{\otimes k}/\langle \sigma\cdot q-q \mid q\in V^{\otimes k}, \sigma\in\Sg_k\rangle$$ for the coinvariants of $V^{\otimes k}$ with respect to the permutation action of $\Sg_k$, and $$\Sym^k(V) := (V^{\otimes k})^{\Sg_k}$$ for the subspace  of $\Sg_k$-invariant tensors in  $V^{\otimes k}$.
Since we are in characteristic 0, thus $\Ryn_k$ factors through the natural map $V^{\otimes k} \to \Sa^k(V)$ to give 
 linear isomorphisms
\beq\label{split2} \xymatrix{
\Sym^k(V) \ar[r]^{\sim} 
& \Sa^k(V) \ar@/^1pc/[l] 
} . \eeq

\subsection{The result of Nagpal and Snowden}
In this subsection we prove Theorem~\ref{ithm:A}(2), which is \cite[Proposition~2.6]{NS}, following the arguments of the unpublished preprint \cite{NS}.
This result is a key tool of this paper.

Let $A$ be a unital commutative ring. 
For $a_1, a_2, \ldots, a_n\in A$ define 
\beq\label{eq:FPol}
\FPol_n(a_1, a_2, \ldots, a_n):=\prod\limits_{1\le i< j\le n}(a_i-a_j).
\eeq
We also set $\FPol_0:=\FPol_1:=1$.
\begin{proposition}[{see~\cite[Proposition~2.6]{NS}}]\label{P:NS}
Let $\mathcal A:=A[z_1, z_2, \ldots]$ be the polynomial ring in countably many variables. The infinite symmetric group $\Sg_\infty$ acts on $\mathcal A$ by
permuting the variables. Let $f\in\mathcal A\setminus(0)$ and let 
\[I_f = \sum\limits_{\sigma\in\Sg_\infty}\mathcal A\sigma(f).\]
 Then there exist $n\ge0$ and a nonzero  $c \in  A $ that is a coefficient of $f$, such that $c\FPol_n(z_1, \ldots, z_n)\in I_f$.
\end{proposition}

Since $I_f $ is the smallest $\Sg_\infty$-invariant ideal of $\mc A$ containing $f$, Theorem~\ref{ithm:A}(2) is an immediate consequence of Proposition~\ref{P:NS}.

\begin{proof}[Proof of Proposition~\ref{P:NS}]
    Our proof is inductive: on the $i$th step we prove that for every unital commutative ring $A$ and every nonzero $f\in A[z_1, \ldots, z_i]$, there exist a nonzero coefficient $c$ of $f$ and $n\ge0$ such that $c\FPol_n(z_1, \ldots, z_n)\in I_f$. 
    The  base case $i=0$ (that is, $f\in A$) is trivial. We also provide a separate argument for the case $i=1$. The general argument works only from step 2 onward.

    Case $i=1$. In this case $f\in A[z_1]$. 
    Write $f=f_0 + f_1 z_1 + \dots + f_d z_1^d$ with $f_j\in A$ and $f_d \neq 0$, so $d = \deg f$. 
    We
   claim that $f_d \FPol_{d+1}(z_1, \dots, z_{d+1}) \in I_f$.
   To see this, let 
%
    \beq\label{eq:F} F=\begin{vmatrix}
         1&1&\ldots&1\\
         z_1&z_2&\ldots&z_{d+1}\\
         z_1^2&z_2^2&\ldots&z_{d+1}^2\\
         \vdots&\vdots&\ddots&\vdots\\
         z_1^{d-1}&z_2^{d-1}&\ldots&z_{d+1}^{d-1}\\
         f(z_1)&f(z_2)&\ldots&f(z_{d+1})\\
    \end{vmatrix}
    = \sum_{j=0}^d f_j \begin{vmatrix}
         1&1&\ldots&1\\
         z_1&z_2&\ldots&z_{d+1}\\
         z_1^2&z_2^2&\ldots&z_{d+1}^2\\
          \vdots&\vdots&\ddots&\vdots\\
         z_1^{d-1}&z_2^{d-1}&\ldots&z_{d+1}^{d-1}\\
         z_1^j&z_2^j&\ldots&z_{d+1}^j\\
    \end{vmatrix}.
    \eeq
    Since $f(z_1), \dots, f(z_{d+1}) \in I_f$, expanding the determinant in the first expression for $F$ along the bottom row we see that $F \in I_f$.
    On the other hand, 
    the determinants in the sum expansion of $F$ on the right hand side of \eqref{eq:F} vanish for $j = 0, \dots, d-1$, and so 
    \[ F = f_d \begin{vmatrix}
         1&1&\ldots&1\\
         z_1&z_2&\ldots&z_{d+1}\\
         z_1^2&z_2^2&\ldots&z_{d+1}^2\\
           \vdots&\vdots&\ddots&\vdots\\
         z_1^{d-1}&z_2^{d-1}&\ldots&z_{d+1}^{d-1}\\
         z_1^d&z_2^d&\ldots&z_{d+1}^d\\
    \end{vmatrix} = \pm f_d \FPol_{d+1}(z_1, \dots, z_{d+1}).\]
      This finishes the case $i=1$.
    
    Now we give the induction step.
    Let $i\ge1$ and let 
     $f\in A[z_1,\ldots, z_{i+1}]$ with $f\ne0$. 
   Since 
    $$\mathcal A \cong A[z_1]\otimes_{A} A[z_2, z_3, \ldots].$$
    we can apply the induction hypothesis with base ring $A[z_1]$. 
    Let $\Sg'_\infty := \Stab_{\Sg_\infty}(z_1) \leq \Sg_\infty$
    and let 
    $I'_f = \sum_{\sigma \in \Sg'_\infty} \mc A \sigma(f)$.
    By induction, there is $c \in A[z_1] \ssm \{0\}$ and $n \geq 0$ so that
    \[ C := c \FPol_n(z_2, \dots, z_{n+1}) \in I'_f.\]
Further, $c$ is a coefficient of $f$ when we think of $f$ as a polynomial in the variables $z_2, \dots, z_{i+1}$ with coefficients in $A[z_1]$.
Note that $I'_f \subseteq I_f$ and that the $A$-coefficients of  $c$ are  a subset of the $A$-coefficients of $f$, up to sign.

Now, $C \in A[z_1, \dots, z_{n+1}] \cong A[z_2, \dots, z_{n+1}][z_1]$.
We apply the $i=1$ case, with coefficient ring $A[z_2, \dots, z_{n+1}]$.
There are thus $m \geq 0$ and a nonzero coefficient $c'$ of $C$ so that 
$c' \FPol_m(z_1, z_{n+2}, \dots, z_{n+m}) \in I_f$.
Here $c'$ is an $A[z_2, \dots, z_{n+1}]$-coefficient of $C = c \FPol_n(z_2, \dots, z_{n+1})$, so 
$c' = c'' \FPol_n(z_2, \dots, z_{n+1})$, where $c''$ is an $A$-coefficient of $c \in A[z_1]$,  and thus, up to sign, an $A$-coefficient of $f$.

From the definition \eqref{eq:FPol} of the determinants $\FPol_k$ and the fact that
\[c'' \FPol_n(z_2, \dots, z_{n+1})\FPol_m(z_1, z_{n+2}, \dots, z_{n+m}) \in I_f,\]
we see that $c'' \FPol_{n+m}(z_1, \dots, z_{n+m}) \in I_f$.
 %
%
%
\end{proof}

\section{Non-noetherianity of the ideal lattice}\label{NONNOETHERIAN}

\subsection{Evaluation homomorphisms}\label{ss:Phi}

We now define the homomorphisms  $\Phi_n$ which will play the starring role in our analysis.
We let $\Phi_1:  \Ua(\Lg) \to \Ua(\mf g)[t, t^{-1}]$ be the ring homomorphism induced from the obvious Lie algebra map 
\[\Lg = \mf g[s, s^{-1}] \to \Ua(\mf g)[t, t^{-1}], \quad g s^i \mapsto g t^i.\]
(Note the change of variable from $s$ to $t$, which will simplify our later notation.)
For $n \geq 1$, define 
\beq\label{Phidef}
\Phi_n = (\Phi_1)^{\otimes n} \circ \Delta^n: \Ua(\Lg ) \to \left( \Ua(\mf g)[t, t^{-1}]\right)^{\otimes n},
\eeq
where $\Delta$ is the coproduct on $\Ua(\Lg)$.
To write $\Phi_n$ explicitly, we use the isomorphism 
\beq\label{direct sum isom}
\left( \Ua(\mf g)[t, t^{-1}]\right)^{\otimes n} \cong \Ua(\mf g^{\oplus n})[t_1^{\pm 1}, \dots, t_n^{\pm 1}].
\eeq
If $g \in \mf g $ let 
\beq \label{eq:component} \begin{matrix}
g_{\{i\}}= 
 &
        &(0,\dots,0,&g,&0,\dots,0) \in \mf g^{\oplus n},\\
        &&&i  \\
\end{matrix}
\eeq
that is with $g$ in the $i$th position. 
 Then in terms of the isomorphism \eqref{direct sum isom} we write
 \[ \Phi_n: \Ua(\Lg) \to \Ua(\mf g^{\oplus n})[t_1^{\pm 1}, \dots, t_n^{\pm 1}]\]
 with 
\begin{equation}\label{EPhi2}
\Phi_n(g s^i)  = g_{\{1\}} t_1^i + \cdots + g_{\{n\}} t_n^i.
\end{equation}

The maps $\Phi_n$ are closely related to  the representation theory of $\Lg$.
If $a\in \kk^\times$, recall that there is an {\em evaluation homomorphism} $\ev_a: \Ua(\Lg) \to \Ua(\mf g)$ defined by $g s^i \mapsto a^i g$.
The maps $\ev_a$ all factor through $\Phi_1$, and in fact $\ker \Phi_1 = \bigcap_{a \in \kk^\times} \ker \ev_a$. 
To see this, for $a \in \kk^\times$ let $\varphi_a: \kk[t, t^{-1}]\to \kk$ be the homomorphism which sends $t \mapsto a$,
so $ev_a = (1\otimes \varphi_a)\circ \Phi_1$.
It follows from the fact that $\bigcap_a \ker \varphi_a =(0)$ that 
\[ \ker \Phi_1 = \Phi_1^{-1} \bigl(\bigcap_a \ker (1 \otimes \varphi_a)\bigr) = \bigcap_a \ker ev_a.\]
We thus refer to $\Phi_1$ as the {\em universal evaluation homomorphism}. 

Evaluation homomorphisms are used in the construction of {\em evaluation representations} of $\Lg$, which we now define.
Let $V$ be a representation of $\mf g$ and let $a \in \kk^\times$.  
The {\em evaluation representation } $V_a$ of $\Lg$ is $\res_{\ev_a}(V)$, i.e., $g s^i$ acts on $V$ by $a^i g$.
We will say that an {\em $n$-fold evaluation representation } is an $n$-fold tensor product
\[ V^1_{a_1} \otimes \cdots \otimes V^n_{a_n}\]
of evaluation representations, where $V^1, \dots, V^n$ are representations of $\mf g$ and $a_1, \dots, a_n \in \kk^\times$.
By \cite{Chari86} (see also \cite[Corollary~6.1]{NSS}), any finite-dimensional irreducible representation of $\Lg$ is an $n$-fold evaluation representation for some $n$.

The following lemma is immediate:

\begin{lemma}\label{lem:Phi evaluation}
For all $n \geq 1$ we have
\[ 
\ker \Phi_n = \bigcap \{ \ \Ann_{\Ua(\Lg)}(\widetilde{V}) \mid \widetilde V \mbox{ is an $n$-fold evaluation representation of $\Lg$ } \}.
\]
\end{lemma}
\begin{proof}
 This follows from the facts that the intersection of the maximal ideals of  $\kk[t_1^{\pm 1}, \dots, t_n^{\pm 1}]$  is trivial,
  and that 
 $0$ is the only element of $\Ua(\mf g)$ which annihilates all  representations of $\mf g$. 
\end{proof}

From Lemma~\ref{lem:Phi evaluation}, we refer to $\Phi_n$ as the {\em  universal $n$-fold evaluation homomorphism}.

\begin{remark}\label{rem:findim}
Lemma~\ref{lem:Phi evaluation} remains true if we restrict to considering {\em finite-dimensional irreducible} $n$-fold evaluation representations of $\Lg$.
This follows from three facts that allow us to improve the proof of Lemma~\ref{lem:Phi evaluation}.

$(1)$ As is well known to experts, if $V^1, \dots, V^n$ are finite-dimensional irreducible representations of $\mf g$ and $a_1, \dots, a_n \in \kk^\times$ are distinct, then $V^1_{a_1} \otimes \cdots \otimes V^n_{a_n}$ is an irreducible representation of $\Lg$; see \cite[(2.3)]{Rao}. 

$(2)$ The intersection 
  \[ \bigcap \{ (t_1 - \lambda_1, \dots, t_n - \lambda_n) \mid \lambda_1, \dots, \lambda_n \in \kk^\times \mbox{ are pairwise distinct}\} \]
  of all ``off-diagonal'' maximal ideals of $\kk[t_1^{\pm 1}, \dots, t_n^{\pm 1}]$  is trivial.

$(3)$    By \cite[Theorem~2.5.7]{Dix},  $0$ is the only element of $\Ua(\mf g)$ which annihilates all finite-dimensional irreducible representations of $\mf g$. 
\end{remark}

\subsection{Infinitely generated ideals}\label{ss:nonnoeth}

It is not hard to see that $\Ua(\Lg)$,
or more generally its deformation $\Ua(\wh{\mf g})/(c-\lambda)$, where $\lambda \in \kk$,
is neither left nor right noetherian: for example, for any $0 \neq g\in \mf g$ the left or right ideal of $\Ua(\wh{\mf g})/(c-\lambda)$ generated by the elements  $\{g s^i \mid i \geq 0\}$ cannot be finitely generated by the PBW theorem.
We now use the homomorphism $\Phi_1$ to prove a  much stronger negative result for $\Ua(\Lg)$ (respectively $\Sa(\Lg)$): it fails to have {\em any} two-sided (respectively  Poisson) noetherianity property.  

\begin{theorem}\label{thm:nonnoeth}
There is an ideal $J$ of $\Ua(\Lg)$ that satisfies the following properties:
\begin{enumerate}
\item  $J^n$ is not finitely generated as a left, right, or two-sided ideal of $\Ua(\Lg)$ for any $n\geq 1$;
\item  $\gr (J^n)$ is not finitely generated as a Poisson ideal of $\Sa(\Lg)$ for any $n\geq 1$;
\item  $(\gr J)^n$ is not finitely generated as a Poisson ideal of $\Sa(\Lg)$ for any $n\geq 1$;
\item  $((\gr J)^n)_{2n}$ is not a finitely generated subrepresentation of $\Sa^{2n}(\Lg)$ for any $n\geq 1$.
\end{enumerate}
\end{theorem}

\begin{proof}
Let $\U = \Ua(\mf g)$, with augmentation ideal $\U_+$, and let $R = \im \Phi_1 \subseteq \U[t,t^{-1}]$.
Let $\U_+[t,t^{-1}]$ denote the ideal of $\U(\mf g)\otimes \kk[t,t^{-1}]$ generated by $\U_+$.
Since $\mf g[t, t^{-1}] \subseteq R$ it is not hard to see that $R= \kk \oplus \U_+  [t, t^{-1}]$.

Let $Z = Z(\U) $ be the centre of $\U$ and let  $\ell$ be the rank of $\mf g$: that is, $\ell$ is the dimension of a Cartan subalgebra of $\mf g$.
Then (see for example \cite[Theorem~7.3.8(ii)]{Dix}) there are algebraically independent $\Omega_1, \dots, \Omega_\ell$, which we can choose to lie in $\U_+$,   so that
$Z = \kk[\Omega_1, \dots, \Omega_\ell]$.
Since the $\Omega_i$ lie in $\U_+$, note that $Z_+:= \U_+\cap Z$ is the ideal of $Z$ generated by $\Omega_1, \dots, \Omega_{\ell}$ and is contained in $R$.  
Furthermore,
\beq\label{RZ} R\cap Z = (\kk \oplus \Ua_+ [t,t^{-1}]) \cap Z = \kk  \oplus Z_+ [t, t^{-1}]. \eeq
Let $K = Z_+ \U[t, t^{-1}]$, and let $J = \Phi_1^{-1}(K)$.

$(1)$, $(2)$.  Let $n \geq 1$.  
We have $\Phi_1(J^n) = K^n = (Z_+)^n\U[t, t^{-1}]$.
For $m \geq 0$, let 
\[K(n, m) = \sum_{i = -m}^m  (Z_+)^n t^i R,\]
so $K^n = \bigcup K(n, m)$.
We can write
\beq\label{eq:K} K(n,m) = (Z_+)^{n-1}\bigl( \sum_{i=-m}^m \sum_{j = 1}^\ell \Omega_j t^i R\bigr).\eeq
Since the $\Omega_j t^i$ are (central) in $R$, thus $K(n,m)$ is an ideal of $R$.
Further, since $t \not\in R$, no $\Omega_k^n t^{m+1}$ is in $K(n,m) $. 
In more detail, $\Ua$ is free over  $Z$  by  Kostant's theorem \cite[Theorem~8.2.4]{Dix}, and in particular $Z$ is a $Z$-summand of $\Ua$.
If $\Omega_k^n t^{m+1} \in K(n,m)$, projecting this to $Z$ by taking the $\mf g$-trivial component of \eqref{eq:K}, we find
\begin{align*}
\Omega_k^n t^{m+1} & \in \sum_{i=-m}^m (Z_+)^n t^i (R\cap Z) \\
 & = \sum_{i=-m}^m (Z_+)^n t^i \oplus (Z_+)^{n+1}[t,t^{-1}] \quad \mbox{ by \eqref{RZ}.}
\end{align*}
This is clearly false.

Thus $\{K(n,m) \}_m$ is an infinite strictly ascending chain of ideals
whose union is $K^n$, showing that $K^n$ is not finitely generated as a two-sided ideal of $R$.
Thus $J^n$ is not finitely generated either.
It follows by standard associated graded arguments that $\gr (J^n)$ is not finitely generated as a Poisson ideal of $\Sa(\Lg)$.

We next prove $(4)$ for $n=1$.
Without loss of generality, we may let  $\Omega := \Omega_1$ be the (unique) {\em quadratic Casimir element} of $\U$:  that is, if 
$\{ g^1, \dots, g^d\}$ is a basis of $\mf g$,
with dual basis $\{ g_1, \dots, g_d\}$ (with respect to the Killing form)
then
\beq\label{casimir}
\Omega = \sum_{i =1}^d  g^i g_i.
\eeq
Let $\omega = \gr \Omega \in \Sa^2(\mf g)$.
The formula for $\omega$ is the same as \eqref{casimir}, but notice that $\Omega \in \U$ and so $g^i , g_i$ may not commute in the formula for $\Omega$.

Now, $\Phi_1$ respects canonical filtrations: that is, $\Phi_1(\Ua_{\leq i}(\Lg)) \subseteq \U_{\leq i}[t,t^{-1}]$ for all $i$ and so $\Phi_1$ induces a graded Poisson homomorphism $
\underline{\Phi}: \Sa(\Lg) \to \Sa(\mf g)[t, t^{-1}]$, where $\underline{\Phi}(g s^i) = g t^i$ for all $g \in \mf g$.
(Here the grading on $\Sa(\mf g)[t, t^{-1}]$ is by giving elements of $\mf g$ degree 1 and giving $t$ degree 0.)
The kernel of $\underline{\Phi}$ is $\gr(\ker \Phi_1)$ and so is contained in the Poisson ideal $\gr J$ of $\Sa(\Lg)$.

Since the elements $\Omega t^i$ are in $R$, we have $\omega t^i \in \underline{\Phi}((\gr J)_2)$ for all $i \in \ZZ$.
These elements are Poisson central and linearly independent, so 
the $\mf g$-trivial component of $\underline{\Phi}((\gr J)_2)$, considered as a representation of $\Lg$ under the adjoint action, is an infinite  direct sum of one-dimensional  representations.
As a result, $\underline{\Phi}((\gr J)_2)$  is not finitely generated.

Now let $n \geq 2$ and consider $((\gr J)^n)_{2n}$.
It is easy to see that $\underline{\Phi}(((\gr J)^n)_{2n}) \supseteq \bigoplus_{i \in \ZZ} \kk \omega^n t^i$, which again is not finitely generated.  
Thus $((\gr J)^n)_{2n}$ is not finitely generated either.

Finally, for $(3)$, if $(\gr J)^n$ were a finitely generated Poisson ideal, its lowest-degree piece $((\gr J)^n)_{2n}$ would be a finitely generated subrepresentation of $\Sa^{2n}(\Lg)$, contradicting $(4)$.
\end{proof}

See Example~\ref{eg:noeth} for more discussion of the ideals $J^n$.

\section{$H$-ideals and $\Theta$-maps}\label{H IDEALS}
The main result of this paper is that, in spite of Theorem~\ref{thm:nonnoeth}, two-sided ideals of $\Ua(\Lg)$ have a great deal of structure:
if $(0) \neq I \normal \Ua(\Lg)$, then $I \supseteq \ker \Phi_n$ for some $n \geq 1$.  
In this section we begin to develop the tools we need for this result, whose proof will take most of the rest of the paper.

\subsection{$H$-ideals}\label{ss:Hideals}
One of our main techniques is to look at $I \cap \Ua(E)$, where $E$ is a suitable infinite-dimensional abelian subalgebra of $\Lg$, and to study its structure.
Let us establish the relevant notation,  which we will use throughout the next two sections.

\begin{notation}\label{not:E}
Let   $\mf b$ be a Borel subalgebra of $\mf g$ and let $\mf h$ be a Cartan subalgebra of $\mf b$.
Let $e\in\mf b$ be a  highest weight vector  for the action of $\mf h$ on $\mf b$,
and let $E:= e\otimes\kk[s, s^{-1}] \subseteq \Lg$.
Since $E$ is an abelian Lie subalgebra of $\Lg$, we can consider $\Sa(E) \cong \Ua(E)$ as a subalgebra both of $\Sa(\Lg)$ and of $\Ua(\Lg)$.
Let $H := \mathrm{L} \mf h = \mf h \otimes \kk[s, s^{-1}]$.
Since $E$ is stable under the adjoint action of $H$, thus $H$ acts on each $E^{\otimes k}$ and $\Sa^k(E)$, and thus also on $\Sa(E)$.
We say that an ideal $I$ of $\Sa(E)$ that is preserved under the adjoint action of $H$ is an {\em $H$-ideal} of $\Sa(E)$, and write 
\[ I \normal_H \Sa(E)\]
to denote this. 

The action of $\mf h$ on $e$ has a kernel of codimension 1 and thus it reduces to the action of a single element $h\in \mf h$ with $[h, e]=e$. 
We set $\Lh= h\otimes \CC[s, s^{-1}]$. 
It is clear that an ideal $I$ of $\Sa(E)$ is $H$-stable if and only if it is $\Lh$-stable. Set $h_{i}:=h\otimes s^i$ and $e_j : = e \otimes s^j$, so $[h_i, e_j] = e_{i+j}$.
Note that $\Sa(E)$ and the $\Lh$-action on it are the same for any choice of a finite-dimensional  simple Lie algebra $\mf g$; in particular, $H$-ideals of $\Sa(E)$ are independent of $\mf g$.

\end{notation}

In this section, we study the structure of $H$-ideals of  $\Sa(E) = \kk[\dots, e_{-1}, e_0, e_1, \dots]$.  
We will see they can be related to symmetrically invariant ideals of a Laurent polynomial ring in countably many variables,
and that the structure of invariant ideals of the infinite polynomial ring also gives structure to $H$-ideals. 
Combining these results, in Section~\ref{S:disc} we will see (Corollary~\ref{Cacc}) that $H$-ideals of $\Sa(E)$ satisfy the ascending chain condition.

The importance of $H$-ideals to our project comes from the observation, following from Theorem~\ref{thm:BS}, that nonzero two-sided ideals of $\Ua(\Lg)$ or Poisson ideals of $\Sa(\Lg)$ give rise to nonzero $H$-ideals of $\Sa(E)$.

\begin{lemma}\label{lem:JtoH} 
Let $J \neq (0)$ be either a two-sided ideal of $\Ua(\Lg)$ or a Poisson ideal of $\Sa(\Lg)$.
Then $J \cap \Sa(E)$ is a nonzero $H$-ideal of $\Sa(E)$.
\end{lemma}
\begin{proof}
    In either case $J$ is preserved by the adjoint action of $H$, and so the same is true for $J \cap \Sa(E)$.

    If $J \normal \Ua(\Lg)$ then by Theorem~\ref{thm:BS}(1) the natural map $\Sa(E) \to \Ua(\Lg)/J$ is not injective and so $J \cap \Sa(E)\neq (0)$.  The proof if $J$ is a Poisson ideal of $\Sa(\Lg)$ is the same, using Theorem~\ref{thm:BS}(2).
\end{proof}

We  remark that the proof of Theorem~\ref{thm:BS} can easily be modified to prove:
\begin{proposition}\label{prop:H-BS}
Let $I$ be a nonzero $H$-ideal of $\Sa(E)$.  Then $\Sa(E)/I$ has polynomial growth.  That is, $\Sa(E)$ has just-infinite growth in terms of $H$-ideals. \qed
\end{proposition} 
We will not need this result, however.

If $r \in \Sa^k(E)$, then $h \cdot r = kr$, where $- \cdot - $ denotes the adjoint action of $H$ on $\Sa(E)$.
It follows that $H$-ideals of $\Sa(E)$ are automatically graded: 
\beq\label{H-graded}
\mbox {if $I \normal_H \Sa(E)$, then }
 I = \bigoplus_{k \geq 0} \gr_k I, 
\mbox{ where $\gr_k I := I \cap \Sa^k(E)$.}
\eeq
Note that $\gr_k I$ is an $H$-subrepresentation of $\Sa^k(E)$.

From \eqref{H-graded}, it makes sense to consider  the structure of $H$-subrepresentations of $\Sa^k(E)$ as a part of our effort to understand $H$-ideals.
To do this, we will use the obvious vector space isomorphism between $\Sa^k(E)$ and a Laurent polynomial ring.
For $k \geq 1$ let  $\mc Z_k:=\kk[z_1^{\pm 1}, \dots, z_k^{\pm 1}]$.
Let 
\[ \Theta^\star_{e^k}: \mc Z_k \to E^{\otimes k}\]
be the linear isomorphism that sends $z_1^{i_1}z_2^{i_2}\dots z_k^{i_k}$ to $e_{i_1}\otimes e_{i_2}\otimes\dots\otimes e_{i_k}$.
(The reason for including $e$ in the notation $\Theta^\star_{e^k}$ is that we will see a more general version of this map later.)
Let 
\[\Theta_{\star}^{e^k} := (\Theta^\star_{e^k})^{-1}: E^{\otimes k} \to \mc Z_k,\] and note that $\Theta^\star_{e^k}$ and $\Theta_\star^{e^k}$ intertwine the natural actions of the symmetric group $\Sg_k$ on $\mc Z_k$ and $E^{\otimes k}$.
Thus $\Theta^\star_{e^k}$ induces a linear isomorphism $\mc Z_k^{\Sg_k} \to \Sa^k(E)$ (using \eqref{split2}), which we denote by $\Theta^{\Sg ; \star}_{e^k}$.
We let $\Theta_{\Sg ; \star}^{e^k} := (\Theta^{\Sg ; \star}_{e^k})^{-1}$, and record the notation in the following display:
\beq \label{Theta} \Theta_{\Sg ; \star}^{e^k} :  \Sa^k(E) \longleftrightarrow \mc Z_k^{\Sg_k} : \Theta^{\Sg ; \star}_{e^k},\eeq
\[e_{i_1}e_{i_2}\dots e_{i_k}
\leftrightarrow 
\Ryn_k(z_1^{i_1} \dots z_k^{i_k}) 
.\]
We refer to the maps $ \Theta^\star_{e^k}$, $\Theta_{\star}^{e^k}$, $\Theta^{\Sg ; \star}_{e^k}$, $\Theta_{\Sg ; \star}^{e^k}$, as well as related maps we will define later, as {\em $\Theta$-maps}.

The reason for introducing the $\Theta$-maps is the following key observation.
\begin{lemma}\label{Lthid}
A subspace $V$ of $E^{\otimes k}$ is an $H$-subrepresentation if and only if $\Theta^{e^k}_{\star}(V)$ is a $\mc Z_k^{\Sg_k}$-submodule of $\mc Z_k$.
Likewise, a subspace $W$ of $\Sa^k(E)$ is an $H$-subrepresentation if and only if $\Theta^{e^k}_{\Sg ; \star}(W)$ is an ideal of  $\mc Z_k^{\Sg_k}$.
\end{lemma}

\begin{proof}
Both parts come from the observation that 
\[ 
\Theta^{e^k}_\star(h_j \cdot \Theta^\star_{e^k}(z_1^{i_1}\dots z_k^{i_k})) 
= \Theta^{e^k}_\star(e_{i_1+j} \otimes \cdots \otimes e_{i_k} +\cdots + e_{i_1} \otimes \cdots \otimes e_{i_k+j})
= z_1^{i_1}\dots z_k^{i_k}(z_1^j + \dots + z_k^j).
\]
Thus $V$ (respectively, $W$) is an $H$-submodule if and only if $\Theta^{e^k}_\star(V)$ (respectively, $\Theta^{e^k}_{\Sg ; \star}(W)$) is closed under multiplication by all $z_1^j + \cdots + z_k^j $.
But the power sums $\{z_1^j + \cdots + z_k^j \mid j \in \ZZ\}$
generate the ring $ \mc Z_k^{\Sg_k}$ of invariant Laurent polynomials.
\end{proof}

It will be helpful to assemble some observations relating ideals of $\mc Z_k$ and ideals of $\mc Z_k^{\Sg_k}$. 
We   define a useful operator that  produces an ideal of $\mc Z_k$ from an  ideal of $\mc Z_k^{\Sg_k}$.
 
\begin{defn}\label{Dceil} If $I \normal \mc Z_k^{\Sg_k}$, define
\[ \lceil I \rceil := \{ p \in \mc Z_k \mid \Ryn_k(p \mc Z_k) \subseteq I \}.\]
\end{defn}

For the next two results, fix $k\geq 1$, and let $\mc Z = \mc Z_k$, $\Sg = \Sg_k$, and $\Ryn = \Ryn_k$.

\begin{lemma}\label{lem:A}
\begin{enumerate}
\item  If $J \normal \mc Z$ then 
$\Ryn(J) = \Bigl( \sum_{\sigma \in \Sg}\sigma(J) \Bigr)
\cap \mc Z^{\Sg}$.

\item There is a bijection 
\[ \{ \mbox{ $\Sg$-invariant radical ideals of $\mc Z$ } \} \to \{ \mbox{ radical ideals of $\mc Z^{\Sg}$ } \} \]
induced by sending $J \mapsto \Ryn(J)$.
\end{enumerate}
\end{lemma}

\begin{proof}
(1)  Let $K = \sum_{\sigma \in \Sg}\sigma(J)$, which is an $\Sg$-invariant ideal of $\mc Z$.
Since $K$ is $\Sg$-invariant,  $\Ryn(K) = K \cap Z^{\Sg}$.  
Further, as $\Ryn (\sigma(p)) = \Ryn (p)$ for all $p \in \mc Z$ and we are in characteristic 0, we have $\Ryn(J) = \Ryn(K)$.

(2) This follows from (1), the Nullstellensatz, and the observation that restriction (that is intersecting with $\mc Z^{\Sg}$) gives  a bijection between maximal $\Sg$-invariant ideals of $\mc Z$ and maximal ideals of $\mc Z^{\Sg}$, as both are in bijection with $\Sg$-orbits in $(\kk^\times)^k$.
\end{proof}

Note that restriction does not give a bijection between $\Sg$-invariant ideals of $\mc Z$ and ideals of $\mc Z^{\Sg}$: 
for example, the ideals of $\mc Z_2$ generated by $z_1-z_2$ and $(z_1-z_2)^2$ both restrict to the same ideal of $\mc Z_2^{\Sg_2}$.

The next lemma gives some elementary properties of the operator $\lceil - \rceil$.

\begin{lemma}\label{Lceil}
Let $I\normal \mc Z^{\Sg}$ and $J \normal \mc Z$.

\begin{enumerate}
\item $\lceil I \rceil$ is an $\Sg$-invariant ideal of $\mc Z$.

\item  $\Ryn(J) \subseteq I$ if and only if $J \subseteq \lceil I \rceil$.

\item $ I = \Ryn(\lceil I \rceil) = \lceil I \rceil \cap \mc Z^{\Sg}$.

\item $\sqrt{\lceil I \rceil} = \sqrt{I \mc Z}$.
\end{enumerate}
\end{lemma}

\begin{proof}
$(1)$, $(2)$ follow from the definition of $\lceil I \rceil$.
For the first equality of $(3)$, 
observe that as $\Ryn$ is $\mc Z^{\Sg}$-linear, $\Ryn(I\mc Z) = I \Ryn(\mc Z) = I$, so $I \subseteq \lceil I \rceil$.
 The second equality is Lemma~\ref{lem:A}(1) combined with $(1)$.

For $(4)$, both
 $\sqrt{\lceil I \rceil}$ and $\sqrt{I \mc Z}$ are $\Sg$-invariant radical ideals of $\mc Z$ whose restriction to $\mc Z^{\Sg}$ is $I$, and the statement follows from Lemma~\ref{lem:A}(2).
\end{proof}

We now consider the action of $\Sg_k \times \Sg_\ell $ on $\mc Z_{k+\ell} \cong \mc Z_k \otimes \mc Z_{\ell}$.

\begin{lemma}\label{lem:proj}
Let $k, \ell \geq 1$, and consider $\mc Z_k, \mc Z_\ell$ as subrings of $\mc Z_k \otimes \mc Z_\ell \cong \mc Z_{k+\ell}$.
If $I \normal \mc Z_k^{\Sg_k}$, then
\[ 
\Ryn_{k+\ell}(I \mc Z_{k+\ell}) = \Ryn_{k+\ell}(I\otimes \mc Z_\ell^{\Sg_\ell}).
\]
\end{lemma}

\begin{proof}
Clearly the right-hand side is contained in the left-hand side.  For the other containment, 
it suffices to show that  for all $a \in I$, $p \in \mc Z_k$, and $q \in \mc Z_\ell$, we have 
  \[ \Ryn_{k+\ell}(ap \otimes q) \in \Ryn_{k+\ell}(I\otimes \mc Z_\ell^{\Sg_\ell}).\]
  By \eqref{proj}, 
 \[ \Ryn_{k+\ell}(ap \otimes q) = \Ryn_{k+\ell}(\Ryn_k(ap) \otimes \Ryn_\ell(q)).\]
 Since $I \subseteq \lceil I \rceil $ by Lemma~\ref{Lceil}(3) and $\Ryn_\ell(q)$ is $\Sg_\ell$-invariant, 
$\Ryn_k(ap) \otimes \Ryn_\ell(q) \in I \otimes \mc Z_\ell^{\Sg_\ell}$, as we need.
  \end{proof}

We now discuss how the bijections \eqref{Theta}  act on $H$-ideals of $\Sa(E)$ as $k$ varies.  
For principal $H$-ideals, it is straightforward, as the following example shows, but in general it is  more complicated.

\begin{example}\label{Esgen}
    Let $k\ge 1$ and let $p\in\mc Z_k^{\Sg_k}$.
    Let $J$ be the $H$-ideal of $\Sa(E)$ generated by $\Theta_{e^k}^{\Sg; \star}(p)$. 
    Then
$\gr_\ell J=(0)$ if $\ell< k$,
  and $\gr_\ell J=\Theta^{\Sg; \star}_{e^\ell}(\Ryn_\ell(p\mc Z_\ell))$ 
  for all $\ell\ge k$. 
\end{example}

\begin{lemma}\label{lem:B} For all $k \geq 0$ let  $I_k$ be an ideal of $\mc Z_k^{\Sg_k}$. 
The following are equivalent:
\begin{enumerate}
\item  $\bigoplus_{k\ge0}\Theta^{\Sg ; \star}_{e^k}(I_k)$ is an $H$-ideal of $\Sa(E)$; 
\item for all $\ell\ge k$ we have $I_k \subseteq \lceil I_\ell \rceil$;
\item for all $\ell\ge k$ we have $\lceil I_k \rceil \subseteq \lceil I_\ell \rceil$.
\end{enumerate}
(Since $\lceil I_k  \rceil \normal \mc Z_k$ and  $\lceil I_\ell \rceil \normal \mc Z_\ell$, in $(2)$, $(3)$ we are using the natural inclusion $\mc Z_k \subseteq \mc Z_\ell$.)
\end{lemma}
\begin{proof}
By Lemma~\ref{Lthid}, all $\Theta^{\Sg ; \star}_{e^k}(I_k)$ are $H$-subrepresentations of $\Sa(E)$, so $(1)$ is equivalent to the statement that
$\bigoplus_{k\ge0}\Theta^{\Sg ; \star}_{e^k}(I_k)$
is an ideal of $\Sa(E)$.
To understand what this means, we must consider how multiplication in $\Sa(E)$ interacts with the $\Theta$-maps.
If $a \in \Sa^k(E)$, $b \in \Sa^j(E)$, then
\beq\label{dag}
\Theta_{\Sg ; \star}^{e^{k+j}}(ab) = \Ryn_{k + j}(\Theta_{\Sg ; \star}^{e^k}(a)(z_1, \dots, z_k) \Theta_{\Sg ; \star}^{e^j}(b)(z_{k+1}, \dots, z_{k + j})),
\eeq
using \eqref{proj}.
So $(1)$ is equivalent to the statement that
\[
I_{k+j} \supseteq \Ryn_{k+j}(I_k \otimes \mc Z_j^{\Sg_j}) = \Ryn_{k+j}(I_k \mc Z_{k+j})
\]
for all $ j , k \geq 1$, where we have used Lemma~\ref{lem:proj}.
This says precisely that $(1)$ $\iff$ $(2)$. 
Then $(2) \iff (3)$ is immediate.
\end{proof}

\subsection{$\Theta$-maps and ideals}\label{ss:Thetaideal}

The compatibility conditions in Lemma~\ref{lem:B} can feel somewhat artificial in terms of the rings $\mc Z_k^{\Sg_k}$, but we can use Lemma~\ref{lem:B} and the $\Theta$-maps to relate $H$-ideals of $\Sa(E)$ to ideals of the infinite Laurent ring
\[ \mc Z_\infty:= \kk[z_i^{\pm 1} \mid i  \geq 1]\]
that are invariant under the natural permutation action of the infinite symmetric group $\Sg_\infty = \bigcup_n \Sg_n$. 
(We refer to such ideals as {\em $\Sg_\infty$-ideals of $\mc Z_\infty$}.)
It turns out that the maps $\Theta^{\Sg ;\star}_{e^k}$ automatically send  an $\Sg_\infty$-ideal to an $H$-ideal, but there are several possible ways to go in the other direction.  

Let us define the relevant notations.
Following our convention with the $\Theta$-maps, we will in general use $(-)_\star$ to denote an object in $\mc Z_\infty$, and $(-)^{\star}$ to denote an object in $\Sa(E)$.

\begin{defn}\label{Dtheta}
Let $I_\star$ be an $\Sg_\infty$-ideal of $\mc Z_\infty$, and let $I^\star$ be an $H$-ideal of $\Sa(E)$.
\begin{enumerate} 
\setcounter{enumi}{-1}
\item Recall from \eqref{H-graded} that  $I^\star$ is automatically a graded ideal of $\Sa(E)$, and that $\gr_k(I^\star) = I^\star \cap \Sa^k(E)$.
\item For all $k$, let $\lyr_k I_\star := I_\star \cap \mc Z_k^{\Sg_k}$.  (Although the subrings $\mc Z_k^{\Sg_k} $ of $\mc Z_\infty$ are not contained in each other, it is useful to think of them as giving a filtration of $\mc Z_\infty$, as the $\mc Z_k$ do.) 
\item We define 
\[\Theta^\star(I_\star):= \bigoplus_{k \geq 0} \Theta^{\Sg ; \star}_{e^k}(\lyr_k I_\star ).\]
\item We define $\Theta_\star(I^\star)$ to be the $\Sg_\infty$-ideal of $\mc Z_\infty$ generated by all $\Theta^{e^k}_{\Sg ; \star}(\gr_k I^\star)$.  
\item We define 
\[ \lceil \Theta \rceil_\star(I^\star) := \bigcup_{k \geq 0} \lceil \Theta^{e^k}_{\Sg ; \star}(\gr_k I^\star ) \rceil. \]
\end{enumerate}
\end{defn}

The next lemma says that these constructions do send ideals to ideals.

\begin{lemma}\label{lem:C}
Let $I_\star$ be an $\Sg_\infty$-ideal of $\mc Z_\infty$, and let $I^\star$ be an $H$-ideal of $\Sa(E)$.
\begin{enumerate}
\setcounter{enumi}{-1}
\item By definition, $\Theta_\star(I^\star)$ is an $\Sg_\infty$-ideal of $\mc Z_\infty$.
\item Further, $\Theta^\star(I_\star)$ is an $H$-ideal of $\Sa(E)$.
\item Finally, $\lceil \Theta \rceil_\star(I^\star)$ is an  $\Sg_\infty$-ideal of $\mc Z_\infty$ and contains $\Theta_\star(I^\star)$.
\end{enumerate}
\end{lemma}
\begin{proof}
$(1)$  Let $\ell \geq k$.
Certainly $(\lyr_k I_\star ) \mc Z_\ell \subseteq I_\star$, so
\[ \Ryn_\ell( (\lyr_k I_\star) \mc Z_\ell) \subseteq \Ryn_\ell(I_\star \cap \mc Z_\ell) = \lyr_\ell I_\star ,\]
where we have used Lemma~\ref{lem:A} for the final equality.
This means by definition that 
$\lyr_k I_* \subseteq \lceil \lyr_\ell I_* \rceil$, 
so the result follows by Lemma~\ref{lem:B}.

$(2)$
First,  $\Theta^{e^k}_{\Sg ; \star}( \gr_k I^\star ) \normal \mc Z_k^{\Sg_k}$ by Lemma~\ref{Lthid}, and thus 
$\lceil \Theta^{e^k}_{\Sg ; \star}( \gr_k I^\star ) \rceil$ makes sense and is an $\Sg_k$-invariant ideal of $\mc Z_k$.
Since $I^\star$ is an $H$-ideal,  by condition $(3)$ of Lemma~\ref{lem:B} we have 
$\lceil \Theta^{e^k}_{\Sg ; \star}( \gr_k I^\star ) \rceil \subseteq \lceil \Theta^{e^\ell}_{\Sg ; \star}( \gr_\ell I^\star ) \rceil$
for all $\ell \geq k$, so the union in the definition of $\lceil \Theta \rceil_\star(I^\star)$ is an increasing union.
It follows that $\lceil \Theta \rceil_\star(I^\star)$ is an $\Sg_\infty$-invariant ideal of $\mc Z_\infty$.  The final statement is then immediate.
\end{proof}

We will see in  Section~\ref{EXAMPLES} that none of
$\Theta_\star, \lceil\Theta\rceil_\star$, $\Theta^\star$ is injective. 
Nevertheless, these operations do behave well with respect to radicals, as the next result shows.

 For an ideal $I^* \normal \Sa(E)$ let 
$$\Sat(I^\star):=\{ r \in \Sa(E) \mid r \Sa^k(E) \subseteq I^\star \mbox{ for $k \gg 0$ } \}$$
be the saturation of $I^\star$ with respect to $\Sa^{\geq 1}(E)$.  
\begin{proposition}\label{Lheqs}
Let $I^\star \normal_H \Sa(E)$ and let $I_\star$ be an $\Sg_\infty$-ideal of $ \mc Z_\infty$.
\begin{enumerate}
\item  $\Sat(I^\star)=\Theta^\star\lceil\Theta\rceil_\star(I^\star)$.
    \item Suppose that $\Sa^d(E)\subseteq I^\star$ for $d\gg0$. Then $\lceil\Theta\rceil_*(I^\star)=\Theta_*(I^\star)=\mc Z_\infty$ and thus $$\Theta^\star\lceil\Theta\rceil_\star(I^\star)=\Theta^\star\Theta_\star(I^\star)=\Sa(E).$$
    \item If the hypothesis of $(2)$ fails, then $\sqrt{I^\star}=\sqrt{\Theta^\star\lceil\Theta\rceil_\star(I^\star)}$.
\item  
$\sqrt{I_\star}=\sqrt{\lceil\Theta\rceil_\star\Theta^\star(I_\star)}$.
\end{enumerate}
\end{proposition}
\begin{proof}
    $(2)$ We have $1= \Theta_{\Sg; \star}^{e^d}(e_0^d) \in\Theta_*(I^*)$. The rest follows.

$(1)$  Let $J_k := \Theta^{e^k}_{\Sg ; \star}(\gr_k I^\star )$ for all $k\ge0$. 
    By definition, $\Theta_\star (I^\star)$ is the $\Sg_\infty$-ideal of $\mc Z_\infty$ generated by $\sum_k J_k$. 
    By Lemma~\ref{lem:B}, for all $\ell \geq k$ we have $\Ryn_\ell(J_k \mc Z_\ell) \subseteq J_\ell$ and $\lceil J_k\rceil\subseteq\lceil J_\ell\rceil$.
It follows that 
\[\Ryn_\ell((\Theta_\star (I^\star) \cap \mc Z_k^{\Sg_k}) \mc Z_\ell) \subseteq \Theta_\star (I^\star) \cap \mc Z^{\Sg_\ell}_\ell\]
for all $\ell \geq k$.
Applying $\Theta^\star$ and \eqref{dag}, we see that $\Theta^\star\lceil\Theta\rceil_\star(I^\star)\subseteq \Sat(I^\star)$.

We now check that 
$\Sat(I^\star)\subseteq \Theta^\star\lceil\Theta\rceil_\star(I^\star)$. 
By \eqref{dag} and Lemma~\ref{lem:proj}, for all $k\ge 0$ we have
$$\Theta^{e^k}_{\Sg ; \star}(\gr_k \Sat(I^\star))=\{p\in\mc Z_k^{\Sg_k}\mid \exists \ell\ge k\colon \forall q\in\mc Z_\ell~\Ryn_\ell(pq)\in J_\ell\}.$$
This implies that for all $k\ge0$ we have 
$\Theta^{e^k}_{\Sg ; \star}(\gr_k \Sat(I^\star)) \subseteq \mc Z_k^{\Sg_k}\cap \lceil\Theta\rceil_\star(I^\star)$ as needed.

The hypothesis of statement (2) fails for $I^\star$ and hence $\Sat(I^\star)\ne\Sa(E)$. 
As $I^\star$ is graded, $\Sat(I^\star)\subseteq \sqrt{I^\star}$ and so the statement follows from  (1). 

$(4)$   By definition,
\[ I_\star \subseteq \bigcup_{k \geq 0} \lceil I_\star \cap \mc Z_k^{\Sg_k} \rceil = \lceil \Theta \rceil_\star \Theta^\star(I_\star).\]
This is contained in $\sqrt{I_\star} $ by Lemma~\ref{Lceil}(4). 
\end{proof}

\section{Discriminants and structure theorems for $H$-ideals} \label{S:disc}

We retain the notation of the previous section, in particular Notation~\ref{not:E} and Definition~\ref{Dtheta}.  In the previous section, we related $H$-ideals of $\Sa(E)$ to $\Sg_\infty$-invariant ideals of the infinite Laurent ring $\mc Z_\infty$.
It is known that $\Sg_\infty$-ideals of $\mc Z_\infty$ have a great deal of structure, and the goal of this section is to derive consequences for $H$-ideals.

\subsection{Finite generation}\label{ss:fingen}

In this subsection we prove that $H$-ideals of $\Sa(E)$ are always finitely generated, that is, $\Sa(E)$ satisfies the ascending chain condition on $H$-ideals.  
We begin with an illustrative example. 
\begin{example}\label{eg:noeth}
Recall from Theorem~\ref{thm:nonnoeth} that $\Phi_1(\Ua(\Lg))$ does not satisfy the ascending chain condition on ideals.
However, we will see later in Corollary~\ref{Cacc} that $\Sa(E)$ does satisfy the ascending chain condition on $H$-ideals.
To make sense of this, let us consider the $H$-structure on $\Phi_1(\Sa(E))$.
Here  $H$ naturally acts on $\kk[e,t, t^{-1}]$ by $h_j = e t^j  \partial_e $, so 
\[ h_j \cdot e^a t^b = a e^{a} t^{b+j},\]
and  $\Phi_1$ intertwines this action with the $H$-action on $\Sa(E)$.
Thus if $I$ is an $H$-ideal of $\Phi_1(\Sa(E)) = \kk \oplus e \kk[e,t, t^{-1}]$, then not only is $I$ closed under multiplication by $e$, but $I \cap e \kk[e,t, t^{-1}]$ is closed under multiplication by all $t^j$ --- in other words, $I \cap e \kk[e,t, t^{-1}]$ is an ideal of $ \kk[e,t, t^{-1}]$.  Effectively, even though   $t \not\in \Phi_1(\Sa(E))$,  the $H$-action allows us to multiply by $t$.

By \eqref{H-graded},  $H$-ideals of $\Phi_1(\Sa(E))$ are graded, where we set  $\deg e =1$, $\deg t=\deg t^{-1} = 0$.
Thus the proper nontrivial  $H$-ideals of $\Phi_1(\Sa(E))$ are 
of the form
$ \bigoplus_{a \geq 1} e^a I_a$,
where $I_1 \subseteq I_2 \subseteq \dots $ are ideals of $\kk[t, t^{-1}]$.
In particular, $\Phi_1(\Sa(E))$ satisfies the ascending chain condition on $H$-ideals, although it is  not noetherian considered as a $\kk$-algebra.

On the other hand, 
recall that the centre of $\Ua(\mf g)[t,t^{-1}]$ is  $Z[t,t^{-1}] = \kk[\Omega_1, \dots, \Omega_\ell, t, t^{-1}]$.
We see that $\Phi_1(\Sa(E)) \cap Z[t,t^{-1}] = \kk$.
It follows that if $K$ is the (infinitely generated) ideal of $\Phi_1(\Ua(\Lg))$ constructed in the proof of Theorem~\ref{thm:nonnoeth} and $J = \Phi_1^{-1}(K)$, then $K^n \cap \Phi_1(\Sa(E)) = (0)$ and $J^n \cap \Sa(E) = (\ker \Phi_1)^n \cap \Sa(E)$ for all $n$.
\end{example}

To study the general structure of $H$-ideals, we begin with the following fundamental result of Cohen.
\begin{theorem}
[\cite{Cohen}]
\label{thm:Cohen}
$\mc Z_\infty$ satisfies the ascending chain condition on $\Sg_\infty$-ideals: that is, any $\Sg_\infty$-ideal of $\mc Z_\infty$ is finitely generated as an $\Sg_\infty$-ideal.
\end{theorem}

We will use this structure and the $\Theta$-maps to prove that $H$-ideals of $\Sa(E)$ also satisfy the ascending chain condition.
However, because the $\Theta$-maps are neither injective nor surjective on ideals, the proof is a little indirect.
We first show that $H$-ideals, which, recall, must be graded, are generated in finitely many degrees.

\begin{theorem}\label{Tacc} Let $I^\star$ be an $H$-ideal of $\Sa(E)$. Then there exists $n \geq 0$ such that
\begin{equation}I^\star=\Sa(E)\gr_0(I^\star)+ \dots+\Sa(E) \gr_n(I^\star).\label{Eacc}\end{equation}\end{theorem}

\begin{remark}
The result says that $I^\star$ is generated in finitely many degrees {\em as an ideal of $\Sa(E)$}.
Note that 
all $\gr_k(I^\star)$ are $H$-subrepresentations,
and that an ideal of $\Sa(E)$ generated by an $H$-subrepresentation is automatically an $H$-ideal. 
\end{remark}
\begin{proof}[Proof of Theorem~\ref{Tacc}]
Consider $I_\star:=\lceil \Theta \rceil_\star(I^\star)$. It is easy to see that $(I_\star\cap\mc Z_n), n\ge 1,$ is a symmetrisation-invariant chain of ideals in the sense of~\cite{AH}: that is, $\sigma \cdot(I_*\cap \mc Z_n)\subseteq I_*\cap \mc Z_m$ for all $m\ge n$ and $\sigma\in\Sg_m$.
This and~\cite[Theorem 4.7]{AH} imply that this chain stabilises modulo the symmetric group:  
that is, there exists $N$ such that for all $m \ge N$
the $\Sg_m$-ideal $I_\star\cap \mc Z_m$ of $\mc Z_m$ is generated by $I_\star\cap\mc Z_N$. 

For all $n\ge 0$ we clearly have $I_\star\cap\mc Z_n=\bigcup_{m\ge n} (\lceil \Theta^{e^{m}}_{\Sg ; \star}(\gr_m I^\star ) \rceil\cap\mc Z_n)$. 
Thanks to Lemma~\ref{lem:B}, $$\lceil \Theta^{e^{m_2}}_{\Sg ; \star}(\gr_{m_2} I^\star ) \rceil\cap\mc Z_n \subseteq \lceil \Theta^{e^{m_1}}_{\Sg ; \star}(\gr_{m_1} I^\star ) \rceil\cap\mc Z_n$$ if $m_1\ge m_2\ge n$. 
Thus $$\lceil \Theta^{e^N}_{\Sg ; \star}(\gr_{N} I^\star ) \rceil\cap\mc Z_N, \lceil \Theta^{e^{N+1}}_{\Sg ; \star}(\gr_{N+1} I^\star ) \rceil\cap\mc Z_N,\dots $$
is an ascending chain of ideals of $\mc Z_N$. 
By the noetherianity of $\mc Z_N$ there exists $N'\ge N$ such that $$I_*\cap\mc Z_N=\lceil \Theta^{e^{N'}}_{\Sg ; \star}(\gr_{N'} I^\star ) \rceil\cap\mc Z_N.$$ 
In this way we attach $N'$ to $N$. 

Let $m \geq N' \geq N$.  Then the $\Sg_m$-ideal of $\mc Z_m$ generated by $I_\star\cap\mc Z_N=\lceil \Theta^{e^{N'}}_{\Sg ; \star}(\gr_{N'} I^\star ) \rceil\cap \mc Z_N$ equals $I_\star\cap\mc Z_m$:
that is $I_\star\cap\mc Z_m$ is the smallest $\Sg_m$-ideal of $\mc Z_m$ containing $\lceil \Theta^{e^{N'}}_{\Sg ; \star}(\gr_{N'} I^\star)\rceil$. 
The latter condition together with Lemma~\ref{lem:B} 
implies that $\gr_m I^\star \subseteq  \Sa(E) \gr_{N'}I^\star$ for all $m\ge N'$. Thus $$I^\star=\Sa(E)\gr_0(I^\star)+ \dots+\Sa(E) \gr_{N'}(I^\star)$$as needed.
\end{proof}

As a corollary, we obtain our first fundamental structure result for $H$-ideals.
Note the contrast with Theorem~\ref{thm:nonnoeth}.

\begin{corollary}\label{Cacc} The lattice of $H$-ideals of $\Sa(E)$ satisfies the ascending chain condition (ACC).\end{corollary}
\begin{proof} Let $I^\star_1\subseteq I^\star_2\subseteq\dots$ be an ascending chain of $H$-ideals in $\Sa(E)$. 
Consider $I^\star:=\bigcup_nI^\star_n$. 
By  Theorem~\ref{Tacc} $I^\star$ is generated by $\gr_0I^\star, \gr_1 I^\star, \dots, \gr_{N}I^\star$ for some $N\ge0$. 
By Lemma~\ref{Lthid} and the noetherianity of $\mc Z_k^{\Sg_k}$ there are $d_0, d_1, \dots, d_N\ge 0$ such that $\gr_k I^\star=\gr_k I^\star_{d_k}$ for all $ 0\le k\le N$. Thus $I^\star=I^\star_n$ for all $n\ge\max_k(d_k)$.
\end{proof}

This proves Theorem~\ref{ithm:B}(1).

\subsection{Discriminants}\label{ss:disc}
To prove Theorem~\ref{ithm:B}(2), we will use Nagpal and Snowden's~\cite[Proposition~2.6]{NS} recent significant strengthening of Cohen's fundamental Theorem~\ref{thm:Cohen}. We reproduced this result in Proposition~\ref{P:NS}. As before consider the {\em determinants}
\beq\label{det} \FPol_k := \prod_{1 \leq i< j \leq k} (z_i -z_j) \in \mc Z_k.\eeq
(As is well known, up to sign
\[ \FPol_k = \begin{vmatrix} 1 & \dots & 1 \\ z_1 & \dots & z_k \\ \vdots & \ddots & \vdots \\ z_1^{k-1} & \dots & z_k^{k-1} \end{vmatrix}\]
is the determinant of a Vandermonde matrix.)
We use the following modification of the result of Nagpal and Snowden.
\begin{theorem}
\label{thm:NS}
Let $I_\star$ be a nonzero  $\Sg_\infty$-ideal of $\mc Z_\infty$.
Then there exists $k \geq 1$ so that $\FPol_\ell \in I_\star$ for all $\ell \geq k$.
\end{theorem}
\begin{proof}
    Choose $f\in I_\star\setminus(0)$. Assume that $f\in\mc Z_i, i\ge0$. Then for some $N$ we have $(z_1z_2\ldots z_i)^Nf\in\mathcal A$ where $\mc A$ is as in Proposition~\ref{P:NS}. By Proposition~\ref{P:NS},  $\FPol_k\in I_\star$ for some $k$. It remains to note that $\FPol_k\mid\FPol_{k+1}$.
\end{proof}
It is useful to understand the determinants $\FPol_k$ more geometrically.  
For all $\ell, k$, let 
\beq\label{diagonal} \Delta^\ell_k := \{ (\alpha_1, \dots, \alpha_\ell) \in (\kk^\times)^\ell \mid \alpha_1, \dots, \alpha_\ell \mbox { take on at most $k$ distinct values }\}\eeq
be the $k$-fold diagonal in $(\kk^\times)^\ell$.
(If $\ell \leq k$ then $\Delta^\ell_k = (\kk^\times)^\ell$.)
Then we have
\begin{theorem}[{\cite[Theorem~2.4]{Lov}}]\label{thm:Lov}
The vanishing ideal 
\[ \I(\Delta^\ell_k) = \{ f \in \mc Z_\ell \mid f \mid_{\Delta^\ell_k} \equiv 0 \} \]
is the $\Sg_\ell$-ideal of $\mc Z_\ell$ generated by $\FPol_{k+1}$.
\end{theorem}

Since by definition $\Delta^\ell_k \subseteq \Delta^\ell_{k+1} \subseteq \dots $, we obtain immediately:
\begin{corollary}\label{cor:Lov}
Let $m \geq \ell, k \geq 1$.
Then $\FPol_{\ell}$ is in the $\Sg_m$-ideal of $\mc Z_m$ generated by $\FPol_k$ if and only if $\ell \geq k$. \qed
\end{corollary}

Observe also that
\begin{equation} \label{Erad}\mbox{ if $p\in \mc Z_k^{\Sg_k}$ and $\FPol_{k}\mid p$, then $ \FPol_{k}^2\mid p$}.\end{equation}
This is important because $\FPol_k$ is not $\Sg_k$-invariant, but $\FPol_k^2$ is.  
We define 
\beq\label{disc} D_k := \Theta^{\Sg ; \star}_{e^k}(\FPol_k^2) \in \Sa^k(E) = \gr_k(\Sa(E)),\eeq
and refer to the $D_k $ as {\em discriminants}.
For example, 
\beq \label{D2}
D_2 = \Theta_{e^2}^{\Sg; \star}\bigl( (z_1-z_2)^2 \bigr) = 2 (e_0 e_2 - e_1^2).
\eeq
Note that $\Phi_1(D_2) = 0$.

We now prove our second fundamental structure result for $H$-ideals.

\begin{theorem}\label{Disc for H ideals} 
Let $I^\star$ be a nonzero $H$-ideal of $\Sa(E)$.
Then there is $k \geq 1$ so that $D_\ell \in I^\star$ for all $\ell \geq k$.
\end{theorem}

\begin{proof}
Consider the $\Sg_\infty$-ideal $\Theta_\star(I^\star)$.
By Theorem~\ref{thm:Cohen}, there is $N$ so that as an $\Sg_\infty$-ideal, $\Theta_\star(I^\star)$ is generated by
\[ \{ \Theta^{e^i}_{\Sg ; \star}(\gr_i I^\star) \mid i \leq N\}.\]
Further, by Theorem~\ref{thm:NS} there is $k$ so that $\FPol_k \in \Theta_\star(I^\star)$, so, possibly increasing $N$, we have
\[ \FPol_k \in \sum_{i \leq N, \sigma \in \Sg_N} \sigma(\Theta^{e^i}_{\Sg ; \star}(\gr_i I^\star) )\mc Z_N.\]
By Corollary~\ref{cor:Lov}, for all $\ell \geq \max(k, N)$, we have
\[ \FPol_\ell \in \sum_{i \leq N, \sigma \in \Sg_\ell} \sigma(\Theta^{e^i}_{\Sg ; \star}(\gr_i I^\star)) \mc Z_\ell.\]

Let $\ell \geq \max(k,N)$.  Then
\[
    \FPol_\ell^2 \in \Ryn_\ell\Bigl(\sum_{i \leq N, \sigma \in \Sg_\ell} \sigma(\Theta^{e^i}_{\Sg ; \star}(\gr_i I^\star)) \mc Z_\ell \Bigr)
     = \sum_{i \leq N} \Ryn_\ell(\Theta^{e^i}_{\Sg ; \star}(\gr_i I^\star) \mc Z_\ell).
\]
By Lemma~\ref{lem:B}, this is contained in $\Theta^{e^\ell}_{\Sg ; \star}(\gr_\ell I^\star)$, and applying $\Theta_{e^\ell}^{\Sg ; \star}$ we obtain that $D_\ell \in \gr_\ell I^\star$.
\end{proof}

\section{A structure theorem for ideals}\label{S:strid}

We have seen that the ideal lattice of $\U(\Lg)$ is non-noetherian.   However, our results on $H$-ideals do give some structure to ideals of $\U(\Lg)$, which we will explore in this section.  

We immediately have:

\begin{theorem}\label{TmainD+} 
Fix $e \in \mf g$ as in Notation~\ref{not:E}, and let $D_1, D_2,\dots$ be the discriminants defined in \eqref{disc}. 
Then each nonzero two-sided ideal of $\Ua(\Lmf g)$ contains all but finitely many of them.
\end{theorem}
\begin{proof}
    This is a direct consequence of Lemma~\ref{lem:JtoH} and Theorem~\ref{Disc for H ideals}.
\end{proof}

Recall the definition of the homomorphisms $\Phi_n:  \Ua(\Lg) \to \Ua(\mf g^{\oplus n})[t_1^{\pm 1}, \dots, t_n^{\pm 1}]$ from \eqref{Phidef}.
The main result of this section is that the ideals $\ker \Phi_n$ are cofinal in the ideal lattice of $\Ua(\Lg)$. 
To prove this, we need to relate the discriminants $D_k$ to the kernels of the $\Phi_n$.

We will need to establish more notation.
\begin{notation}\label{not:grade}
Recall that  $d=\dim\mf g$ and let $b_1, b_2, \dots, b_d$ be a basis of $\mf g$.
This choice of basis defines $\NN^d$-gradings on $\Ua(\Lg)$, on $\Ua(\mf g)$, and on $\Ua(\mf g^{\oplus n})[t_1^{\pm 1}, \dots, t_n^{\pm 1}]$ as follows.
For $1\leq i\leq d$ and
 $k \in \ZZ$ let
\[ b_{i[k]} = b_i s^k \in\Lg.\]
For $1\leq j\leq n$ let $b_{1\{j\}}, \dots, b_{d\{j\}}$ be the elements corresponding to $b_1, \dots, b_d$ in the $j$th copy of $\mf g$ in $\mf g^{\oplus n}$, as in \eqref{eq:component}.

As vector spaces, 
\[ \Ua(\Lg) \cong \kk[b_{i[k] }\mid 1 \leq i \leq d, k \in \ZZ]\]
and
\[ \Ua(\mf g^{\oplus n}) \cong \kk[b_{i\{j\}} \mid 1 \leq i \leq d, 1 \leq j \leq n],\]
by the PBW theorem.
Let \[\epsilon_1:=(1, 0, \dots, 0), \epsilon_2:=(0, 1, 0, \dots, 0), \dots, \epsilon_d:=(0, \dots, 0, 1) \in \ZZ^d.\]
Define $\ZZ^d_{\geq 0}$-gradings on $\Ua(\Lg)$ and on $\Ua(\mf g^{\oplus n})$ by setting
\[ \deg b_{i[k]} = \deg b_{i\{j\}} = \epsilon_i.\]
For $q\in\NN^d$ we let $\Ua(\Lmf g)_q$  be the respective graded component, and similarly for the other enveloping algebras.
\end{notation}

The gradings defined by Notation~\ref{not:grade} are vector space gradings, not algebra gradings, but nevertheless it is immediate from the definition that $\Phi_n$ respects the graded structure: for any basis $\{b_1, \dots, b_d\}  $ of $\mf g$ we have 
\beq\label{qgradeA}
\Phi_n(\Ua(\Lg)_q) \subseteq \Ua(\mf g^{\oplus n})_q[t_1^{\pm 1}, \dots, t_n^{\pm 1}]
\eeq
for all $ q\in \NN^d$.
Thus
\beq\label{qgradeB}
\ker \Phi_n = \bigoplus_{q \in \ZZ^d_{\geq 0}}(\ker \Phi_n)_q.
\eeq

Our  goal is to describe the components $(\ker \Phi_n)_q$.  
To do this, we  will need more general versions of the $\Theta$-maps, and we now set up notation.
Let $q=(q_1, \dots, q_d) \in\NN^d$. Let  $\hat q_1:=0$, $\hat q_i:=q_1+q_2+\dots+q_{i-1}$, and $ |q|:=\hat q_{d+1} =  q_1+\dots+q_d$. Put $\mc Z_q:=\mc Z_{|q|}$. We denote by $\Sg_q$ the group $\Sg_{q_1}\times \Sg_{q_2}\times\dots$ where $\Sg_{q_1}$ permutes $z_1, z_2, \dots, z_{q_1}$, $\Sg_{q_2}$ permutes $z_{\hat q_2+1}, \dots, z_{\hat q_2+q_2}$ and so on.  
Clearly,  $\Sg_q\subseteq\Sg_{|q|}$. 
Let $\Ryn_q$ denote the Reynolds operator $\Ryn_{\Sg_q}:\mc Z_q\to\mc \mc Z_q^{\Sg_q}$. 
   
Let $b^q:=b_1^{q_1} \dots b_d^{q_d}$. 
By a clear analogy with \eqref{Theta}, for $q\in \NN^d$ we 
identify the vector spaces $\Ua(\Lmf g)_q$ and $ \mc Z_q^{\Sg_q}$ via linear maps
\beq \label{Theta2} \Theta_{\Sg ; \star}^{b^q}: \Ua(\Lmf g)_q  \longleftrightarrow    \mc Z_q^{\Sg_q} : \Theta^{\Sg ; \star}_{b^q}
 ,\eeq
\beq\label{E:horf} 
\prod\limits_{i=1}^d b_{i[c_{1, i}]}b_{i[c_{2, i}]}\dots b_{i[c_{q_i, i}]}
\leftrightarrow
\Ryn_{q} \Bigl( \prod_{i=1}^d z_{\hat q_i+1}^{c_{1, i}}z_{\hat q_i+2}^{c_{2, i}}\dots z_{\hat q_i+q_i}^{c_{q_i, i}}\Bigr).
\eeq
The notation of (\ref{Theta}) is compatible with (\ref{Theta2}) in the following sense:
if we choose the basis $b_1, \dots, b_d$ of $\mf g$ in such a way that $e=b_1$, then  $\Theta^{\Sg; \star}_{e^\ell}=\Theta^{\Sg; \star}_{b^{(\ell, 0, \dots, 0)}}$. 

We will compute the kernel of $\Phi_n$ below, but first let us consider an example.

\begin{example}\label{eg:kerphi}
Let $\mf g = \mf{sl}_2$ and consider $ \ker \Phi_1$.
Let $e,h,f$ be the standard basis of $\mf sl_2$, 
and write $e_i = e s^i, h_i = hs^i, f_i = fs^i$.
It is easy to see that if $i+j=k+\ell$ then $x_iy_j-x_ky_\ell\in\ker\Phi_1$ for all $x, y\in\{e, h, f\}$.
Denote by $J$ the ideal of $\Ua(\LslT)$ generated by the above elements. 
We claim that $J = \ker \Phi_1$.

We have observed that  $J\subseteq\ker\Phi_1$. 
On the other hand, if $a \geq 1$ then 
$$e_{i_1}e_{i_2}\dots e_{i_a}h_{j_1}h_{j_2}\dots h_{j_b}f_{k_1}f_{k_2}\dots f_{k_c}\equiv e_{i_1+\dots+i_a+j_1+\dots+j_b+k_1+\dots+k_c}e_0^{a-1}h_0^bf_0^c\mod J.$$
This  and similar relations show that $\Ua(\LslT)/J$ is spanned by the images of the  elements
$$1, e_m e_0^ah_0^bf_0^c, a, b, c\in \NN, m \in\ZZ, \hspace{10pt} h_m h_0^bf_0^c, b, c\in \NN, m\in\ZZ, \hspace{10pt} f_m f_0^c, c\in \NN, m\in\ZZ.$$
Under $\Phi_1$, these map respectively to
\[ 1, e^{a+1}h^b f^c t^m, h^{b+1} f^c t^m, f^{c+1} t^m,\]
which give a basis of $\im \Phi_1$. 
Thus these elements are also a basis of $\Ua(\LslT)/J$: in other words, $\Ua(\LslT)/J\to \im \Phi_1$ is an isomorphism and hence $J=\ker\Phi_1$.

A similar analysis shows that the discriminant $D_2$ does not generate $\Ker \Phi_1$, although from \eqref{D2} we know that  $D_2 \in \Ker \Phi_1$.
\end{example}

We  now describe the kernel of $\Phi_n$.   

\begin{proposition}\label{prop:kernel}
Fix $q \in \NN^d$ and let $n \in \ZZ_{\geq 1}$.  
Then 
\[ \ker \Phi_n \cap \Ua(\Lg)_q = 
\Theta^{\Sg ; \star}_{b^q}\Bigl(
\I(\Delta_n^{|q|}) \cap \mc Z_q^{\Sg_q}\Bigr),\]
where $\Delta^{|q|}_n$ is the $n$-fold diagonal defined in \eqref{diagonal}.
In particular, if we let $e \in \mf g$ following Notation~\ref{not:E}, then  $\Phi_n(D_{n+1}) = 0$.
\end{proposition}

Before the proof, we give an example  which may be helpful.
\begin{example}
Let $\mf g = \mf{sl}_2$ with standard basis $e,h,f$.
As before, let $e_i := e s^i \in \mathrm{L}\mf{sl}_2$.
Let $p = z_1z_2$.
Then 
\begin{multline*}
\Phi_2(\Theta_{e^2}^{\Sg;\star}(p)) = \Phi_2(e_1^2)
 = (e_{\{1\}} t_1 + e_{\{2\}} t_2)^2 
 = (e_{\{1\}})^2 t_1^2 + 2 e_{\{1\}} e_{\{2\}} t_1 t_2 + (e_{\{2\}})^2 t_2^2 \\
 = 
 (e_{\{1\}})^2 p(t_1, t_1) +  e_{\{1\}} e_{\{2\}}(p( t_1, t_2) + p(t_2, t_1))  +( e_{\{2\}})^2 p(t_2, t_2).
\end{multline*}
\end{example}

\begin{proof}[Proof of Proposition~\ref{prop:kernel}] It might be easier to write than to read this proof. 

Let's consider a single monomial on the left-hand side of ~\eqref{E:horf}. 
We have 
$$\Phi_n(\prod_{i=1}^d \prod_{j=1}^{q_i} b_{i[{c_{j, i}}]})=
\prod_{i=1}^d \prod_{j=1}^{q_i} (b_{i\{1\}}t_1^{c_{j, i}}+b_{i\{2\}}t_2^{c_{j, i}}+\dots+ b_{i\{n\}}t_n^{c_{j, i}}).$$
Combining this with~\eqref{E:horf}, for $p \in \mc Z_q^{\Sg_q}$ we have
$$\Phi_n(\Theta^{\Sg; \star}_{b^q}(p))=\sum\limits_{s_1, \dots, s_{|q|}\in\{1, \dots, n\}}
\Bigl(p(t_{s_1}, \dots, t_{s_{|q|}})\prod\limits_i b_{i\{s_{\hat q_i+1}\}}b_{i\{s_{\hat q_i +2}\}}\dots b_{i\{s_{\hat q_i+q_i}\}}\Bigr).$$
Since $p\in\mc Z_q^{\Sg_q}$, 
when we collect terms the (polynomial) coefficient  of 
$\prod\limits_i b_{i\{s_{\hat q_i+1}\}}b_{i\{s_{\hat q_i+2}\}}\dots b_{i\{s_{\hat q_i+q_i}\}}$ 
is always a positive integer multiple of 
$p(t_{s_1}, \dots, t_{s_{|q|}})$. 
Thus all the coefficients vanish if and only if $p(t_{s_1}, \dots, t_{s_{|q|}})=0$ for all choices of $s_i$,
in other words, if and only if $p$ vanishes on $\Delta^{|q|}_n$. 
(Here we use that $\Delta^{|q|}_n \supseteq \Delta^{|q|}_{n-1} \supseteq \dots$.)

The last statement follows from Theorem~\ref{thm:Lov}.
\end{proof}

\begin{corollary}\label{cor:kernew}
\begin{enumerate}
    \item If $|q|\leq n$ then $(\ker \Phi_n)_q = (0)$. 
    \item We have
    \[ \bigcap \{ \Ann_{\U(\Lg)}(V) \mid \mbox{ $V$ is a finite-dimensional irreducible representation of $\Lg$ } \} = \bigcap_n \ker(\Phi_n) = (0).\]
\end{enumerate}
\end{corollary}

\begin{proof} 
$(1)$  is because $\Delta_n^{|q|} =(\kk^\times)^{|q|}$.

$(2)$  The second equality follows from  $(1)$ and \eqref{qgradeB}.
By Remark~\ref{rem:findim}, 
\[ 
\bigcap \{ \Ann_{\U(\Lg)}(V) \mid \mbox{ $V$ is an irreducible finite-dimensional representation of $\Lg$ } \}
\]
is contained in $\bigcap_n \ker \Phi_n$ and the first equality follows.
\end{proof}

\begin{remark}
Suppose that $|q|=n+1$.
Note that $\Sg_{n+1}$ acts on $\FPol_{n+1}$ by the sign representation.
It follows from the discussion after Theorem~\ref{thm:NS} that $\I(\Delta_n^{n+1}) \cap \mc Z_{n+1}^{\Sg_q}$ is the ideal of $\mc Z_q^{\Sg_q}$ generated by
 \[  \prod_{i < j \in \{1, \dots, n+1\}} (z_i-z_j)^{\gamma_{ij}},\]
 where 
 \[ \gamma_{ij} = \begin{cases} 2 & \mbox{ if $i,j $ are in the same  $\Sg_q $-orbit } \\ 1 & \mbox{ else.}\end{cases}\]
 (Here the action of $\Sg_q$ on $\{1, \dots, n+1\}$ is via the embedding $\Sg_q \subseteq \Sg_{|q|}$.)
\end{remark}

\begin{notation}\label{not:E2}
The description of $\Ker \Phi_n$ from Proposition~\ref{prop:kernel} is especially nice if $b_1 = e$ is  a highest weight vector for some Cartan subalgebra of $\mf g$ and $q = (n,0, \dots, 0)$, that is, when $\Ua(\Lg)_q = \Sa^n(E)$, where as before $E = e\otimes \kk[s,s^{-1}] \subseteq\Lg$.  We will abbreviate this situation by saying that {\em $b_1 = e$ is a highest weight vector} if $b_1$ is a highest weight vector for some Cartan subalgebra of $\mf g$.
If $e$ is a highest weight vector, we always let $E = e\otimes \kk[s,s^{-1}]$.
\end{notation}

\begin{lemma}\label{Lfpolgen}
Let $n \geq 1$ and $1 \leq i\leq d$.  Fix a basis $b_1,\dots, b_d$ of $\mf g$ so $b_i = e$ is a highest weight vector.  
Recall $\FPol_{n+1} \in \mc Z_{n+1}$ from \eqref{det} and let $D = D_{n+1} \in \Sa^{n+1}(E)$ be the $(n+1)$st discriminant \eqref{disc}.
 Let  $J$ be the ideal of $\Ua(\Lg)$ generated by $D$, and let $I$ be the $H$-ideal of $\Sa(E)$ generated by $D$.
\begin{enumerate}
   \item For all $\ell \geq n+1$ we have
   \[ \Theta^{\Sg;\star}_{e^\ell}(\Ryn_\ell(\mc Z_\ell \FPol_{n+1})) \subseteq \gr_\ell(I).\]
   \item For all $q\in \NN^d$ with $q_i\geq n+1$, 
   \[ \Theta^{\Sg;\star}_{b^q}(\Ryn_q(\mc Z_q \FPol_{n+1}(z_{\hat{q_i}+1}, \dots, z_{\hat{q_i}+n+1}))) \subseteq J.\]
   \item  $I = \ker \Phi_n \cap \Sa(E)$.
\end{enumerate}
\end{lemma}
\begin{proof}
   $(1)$ Let $p \in \mc Z_\ell$.  Since $\Ryn_{n+1}(p \FPol_{n+1})$ is $\Sg_{n+1}$-invariant, it is divisible by $\FPol_{n+1}^2$ by \eqref{Erad}.
   Thus
  \[
   \Ryn_\ell(\mc Z_\ell \FPol_{n+1}) = \Ryn_\ell(\mc Z_\ell \FPol_{n+1}^2).
   \]
   By \eqref{dag} and Lemma~\ref{lem:proj}, 
   \[ \Ryn_\ell(\mc Z_\ell \FPol_{n+1}^2) \subseteq \Theta^{e^\ell}_{\Sg;\star}(\gr_\ell I).\]

  The proof of  $(2)$ is similar.  

  $(3)$ now follows from Proposition~\ref{prop:kernel}. 
\end{proof}

We immediately obtain:

\begin{corollary}\label{CordD} 
Let $k, \ell \geq 1$.  
Let $e$ be  a highest weight vector and consider the discriminants $D_{\ell+1}, D_{k+1} \in \Sa(E)$.
The following are equivalent:
\begin{enumerate}
   \item  $\ell \geq k$;
   \item $D_{\ell+1} \in \ker \Phi_k$;
   \item $\Ker \Phi_\ell \subseteq \ker \Phi_k$;
   \item $D_{\ell+1}$ is in the ideal of $\Ua(\Lg)$ generated by $D_{k+1}$.
\end{enumerate}
\end{corollary}
\begin{proof}
This follows from Lemma~\ref{Lfpolgen}, Proposition~\ref{prop:kernel}, and Corollary~\ref{cor:Lov}.
\end{proof}

We have seen in Proposition~\ref{prop:kernel} and Example~\ref{eg:kerphi} that although $D_{k+1} \in \Ker \Phi_k$, in general $D_{k+1}$ does not generate $\ker \Phi_k$.  However, we do have the following result.

\begin{proposition}\label{Pkto3k}
Let $e \in \mf g$ be a highest weight vector.
Pick $k\ge 1$ and let $J$ be the ideal of $\Ua(\Lmf g)$ generated by $D_{k+1} \in \Sa^{k+1}(E)$.  Then $J \supseteq \Ker \Phi_{dk}$.\end{proposition}

\begin{proof}
    We pick a basis $b_1, \dots, b_d$ of $\mf g$ such that $b_1=e$ and  each $b_i$ is a highest weight vector for a Borel subalgebra $\mf b_i$ of $\mf g$. 
    Since  the adjoint action of $\mf g$ on $\Ua(\Lmf g)$ is rational, it gives rise to the adjoint action of a  semisimple algebraic group $G$ whose Lie algebra is isomorphic to  $\mf g$. 
    Such a group $G$ acts transitively on  the set of Borel subalgebras of $\mf g$.
    Thus for all $i$ there is $g_i \in G$ so that $g_i e = b_i$. 
     
    For all  $i$, we have $\Theta$-maps which identify
    \[\Theta^{b^{(k+1) \epsilon_i}}_{\Sg;\star}: \Ua(\Lg)_{(k+1)\epsilon_i}
 \leftrightarrow 
     \mc Z_{k+1}^{\Sg_{k+1}} :\Theta^{\Sg; \star}_{b^{(k+1)\epsilon_i}}
   .\]
    Recall that $D_{k+1} = \Theta_{e^{k+1}}^{\Sg; \star}(\FPol^2_{k+1})$.
    Because every two-sided ideal of $\Ua(\Lmf g)$ is stable under the adjoint action of $\mf g$, 
    it is also stable under the natural action of $G$.
    Applying the $g_i$, we see  that for all $i$ the two-sided ideal of $\Ua(\Lmf g)$ generated by $\Theta_{b_i^{k+1}}^{\Sg; \star}(\FPol_{k+1}^2)$ is equal to $J$. 
    
Thanks to Proposition~\ref{prop:kernel}, for any $q \in \NN^d$  we have  $\Ker \Phi_{dk}\cap \Ua(\Lg)_{q}=\Theta^{\Sg; \star}_{b^q}(\I(\Delta_{dk}^{|q|})\cap\mc Z_q^{\Sg_q})$, and this intersection is nonzero if and only if $|q| \geq dk+1$.
    Let $q \in \NN^d$ with $|q| \geq dk+1$ and let $r\in\ker\Phi_{dk}\cap \Ua(\Lg)_{q}$. 
    Then
    $$\Theta^{b^q}_{\Sg; \star}(r)=\sum\limits_{1\le i_1<\dots<i_{dk+1}\leq|q|}\FPol_{dk+1}(z_{i_1}, \dots, z_{i_{kd+1}})p_I$$
    for some $p_I\in\mc Z_{q}$.
    Note that the left-hand side belongs to $\mc Z_q^{\Sg_q}=\im \Ryn_q$. 
    Thus it is enough to show that $$\Theta_{b^q}^{\Sg; \star}(\Ryn_q(\FPol_{dk+1}(z_{i_1}, \dots, z_{i_{kd+1}})p_I))\in J.$$
    
Recall the notation $\hat q_i = q_1 + \dots + q_{i-1}$.
    By the pigeonhole principle there exists $1 \leq \beta \leq d$ such that the interval $[\hat q_\beta+1, \hat q_{\beta+1}]$ contains at least $k+1$ numbers out of $i_1<i_2<\dots<i_{dk+1}$.    
    Let them be $i_{w_-}, \dots, i_{w_+}$. 
    Then $\FPol_{k+1}(z_{i_{w_-}}, \dots, z_{i_{w_+}})$ divides $\FPol_{dk+1}(z_{i_1}, \dots, z_{i_{dk+1}})p_I$. 
    This and Lemma~\ref{Lfpolgen} imply the desired result.
\end{proof}

Theorem~\ref{TmainD+} and Proposition~\ref{Pkto3k} immediately imply the following structure theorem for ideals of $\Ua(\Lg)$. 
\begin{theorem}\label{TmainI+} Each nonzero two-sided ideal  of $\Ua(\Lmf g)$ contains all but finitely many of the ideals $\Ker\Phi_k$. \qed \end{theorem}

\begin{remark}\label{rmk:fingen}
We believe the ideals $\ker \Phi_k$ are all finitely generated, although we have not been able to prove it in general.  This is the subject of ongoing research.
\end{remark}

\section{Annihilators of finite-dimensional representations}\label{REPRESENTATIONS}

In this section, we apply our results to ideals of $\Ua(\Lg)$ coming from representation theory.
From Theorem~\ref{TmainI+}, we immediately obtain the following corollary.

\begin{corollary}\label{cor:findim}
  Let $V$ be    a finite-dimensional representation of $\Lg$.
  Then there is some $k \geq 1$ so that $\ker \Phi_k \subseteq \Ann_{\U(\Lg)}(V)$.  \qed
\end{corollary}  
While this follows  from Theorem~\ref{TmainI+}, a direct proof can also be given using the techniques of this paper.

\begin{proof}[Direct proof of Corollary~\ref{cor:findim}]
Fix a Borel subalgebra $\mf b$ of $\mf g$, a Cartan subalgebra $\mf h$ of $\mf b$,  a highest weight vector $e \in \mf b$,
and $h \in \mf h$ so that $[h, e] = e$. 
Recall  that  $E= e\otimes\kk[s, s^{-1}]$ is an abelian subalgebra of $\Lg$.

Let $J = \Ann_{\U(\Lg)}(V)$ and let $I = J \cap \Sa(E)$.
Since $V$ is finite-dimensional, it has a highest $h$-weight, so $I \supseteq \Sa^{\geq \ell}(E)$ for some $\ell$. 
Thus $I$ contains the discriminants $D_{\ell'}$ for $\ell' \geq \ell$.
By Proposition~\ref{Pkto3k}, $J \supseteq \ker \Phi_{d\ell}$, where recall $d = \dim \mf g$.
\end{proof}

For a non-faithful representation $V$ of $\Lg$, the minimum $k$ so that $\ker \Phi_k \cdot V = 0$ is an interesting invariant of $V$.  
The next result gives another, often much better, upper bound for this invariant.

Let $\mf k$ be an ideal of $\Lmf g$ with $\dim\Lmf g/\mf k<\infty$; we set $\bar{\mf k}=\Lmf g/\mf k$. Recall that in this case we have $\mf k=p(s)\Lmf g$ for a unique monic polynomial $p(s)\in\CC[s]$ with $p(0) \neq 0$, and also $\bar{\mf k}=\mf g\otimes (\CC[s,s^{-1}]/(p(s)))$. 
Let $k :=\deg p(s)$, so $\dim \bar{\mf k} = dk$. 
We wish to show the following.
\begin{proposition}\label{prop:annihilator2}
     The ideal of $\Ua(\Lmf g)$ generated by $\mf k$ contains $\Ker \Phi_{k}$.
\end{proposition} 
\begin{proof} It is enough to check that the natural map $\pi: \Ua(\Lmf g)\to \Ua(\bar{\mf k})$ sends $\Ker \Phi_{k}$ to 0. 
Fix a basis $\{b_1,\dots, b_d\}$ of $\mf g$ and consider the $\NN^d$-grading on $\Ua(\Lg)$ from Notation~\ref{not:grade}.
Recall from \eqref{qgradeB} that
$$\Ker \Phi_{k}=\bigoplus_{q\in\ZZ_{\ge0}^d}(\Ker \Phi_{k}\cap\Ua(\Lmf g)_q)$$
and also recall the definitions of the   linear bijections
$$  \Theta_{\Sg;\star}^{b^q}: \Ua(\Lmf g)_q  \leftrightarrow  \mc Z_q^{\Sg_q}: \Theta^{\Sg;\star}_{b^q}.$$
We will suppress some notation and simply refer to the bijections above as $\Theta$-maps.
Likewise,  there is a decomposition
\[ \Ua(\bar{\mf k}) = \bigoplus_q \Ua(\bar{\mf k})_q,\]
where $\Ua(\bar{\mf k})_q = \pi(\Ua(\Lg)_q)$.
There are also natural bijections, which we again refer to as $\Theta$-maps, between 
\[ \Ua(\bar{\mf k})_q \leftrightarrow \Bigl( \mc Z_q/\langle p(z_1), \dots, p(z_{|q|})\rangle \Bigr)^{\Sg_q}.\]
(Here we use $\langle p(z_1), \dots, p(z_{|q|})\rangle$ to denote the ideal of $\mc Z_q$ generated by $p(z_1), \dots, p(z_{|q|})$.)
Since $\pi$ respects these decompositions,  it suffices to prove that $(\ker \Phi_k)_q \subseteq \ker \pi$ for all $q$.

Fix $q\in \NN^d$, and set $I_{k}$ to be the $\Sg_{|q|}$-ideal of $\mc Z_q$ generated by $\FPol_{k+1}$, where recall that
\[ \FPol_{k+1}=\prod_{1\le i<j\le k+1}(z_i-z_j).\]
Thanks to Proposition~\ref{prop:kernel} and Theorem~\ref{thm:Lov}, 
 $\Ker \Phi_{k}\cap\Ua(\Lmf g)_q$ is identified under the $\Theta$-maps with  $\mc Z_{q}^{\Sg_q}\cap I_{k}$. 
Further, 
 the  $\Theta$-maps identify
\[ \ker \pi \cap \Ua(\Lg)_q \leftrightarrow  \langle p(z_1), \dots, p(z_{|q|})\rangle \cap \mc Z_q^{\Sg_q}.\]
To prove that $\pi$ sends $(\ker \Phi_k)_q$ to $0$, it thus suffices to show that $I_k \subseteq \langle p(z_1), \dots, p(z_{|q|})\rangle$.

To prove this,  we consider a formal polynomial
$$\tilde p(u):=u^{k}+p_{k-1}u^{k-1}+\dots+p_1u+p_0$$
where $p_0, p_1, \dots$ are free parameters. 
We set 
$T:=\CC[p_0, p_1, \dots, p_{k-1}]$ and let $ \tilde{\CC}:=\CC(p_0, p_1, \dots, p_{k-1})$ be  the quotient field of $T$. 

 As the next step we need the following.

\begin{lemma} \label{lem:formalfoo}
The ideal $I_{k}\otimes T\subseteq\mc Z_{q}\otimes T$ is contained in the ideal $\langle \tilde p(z_1),\dots, \tilde p(z_{|q|})\rangle$ of $\mc Z_{q}\otimes T$. \end{lemma}
\begin{proof}
Note that both $\mc Z_{q}\otimes T$ and $(\mc Z_{q}\otimes T)/\langle \tilde p(z_1), \dots, \tilde p(z_{|q|})\rangle$ are flat  $T$-modules: the first  is free  and the second one is an injective limit of a sequence of free $T$-modules. 
In particular both are $T$-torsionfree.
Thus the needed inclusion is equivalent to showing that $I_{k}\otimes \tilde{\CC}\subseteq\langle \tilde p(z_1),\dots, \tilde p(z_{|q|})\rangle\otimes_T \tilde{\CC}$. 

Set $\hat \kk$ to be the algebraic closure of $\tilde{\CC}$. 
It is clear that it is enough to prove the analogue of the desired statement over $\hat{\CC}$.

To see that the statement holds over $\hat{\CC}$, note we have $$\tilde p(u)=(u-u_1)(u-u_2)\dots(u-u_{k})$$ for some pairwise distinct $u_1,  \dots u_k \in\hat{\CC}$ and moreover 
$$(\tilde{\CC}[u]/(\tilde p(u)))\otimes_{\tilde\CC}\hat{\CC}\cong \hat{\CC}[u]/(\tilde p(u))\cong\bigoplus_{i=1}^{ k} \hat{\CC}[u]/(u-u_i)\cong \bigoplus_{i=1}^{ k} \hat{\CC}.$$ 
Thus we have
\beq\label{formal}
\frac{\mc  Z_q\otimes\hat{\CC}}{\langle \tilde p(z_1),\dots, \tilde p(z_{|q|})\rangle} \cong
\bigoplus \limits_{i_1, i_2, \dots, i_{|q|}\in\{1, 2, \dots, k\}} \frac{\mc Z_{q}\otimes\hat{\CC}}{\langle z_1-u_{i_1}, z_2-u_{i_2}, \dots, z_{|q|}-u_{i_{|q|}}\rangle}.
\eeq
But in every component of the right-hand side of \eqref{formal}, the image of any polynomial in $I_k\otimes \hat{\kk}$ 
must vanish, as $z_1, \dots, z_{|q|}$ can certainly take on no more than $k$ distinct values.
This proves the claim.
\end{proof}

We recall that $p(u)$ is a monic polynomial of degree $k$ and thus $p(u)=u^k+\alpha_{k-1}u^{k-1}+\dots+\alpha_0$ for some $\alpha_i\in\CC$. We denote by $\mf m$ the ideal of $T$ generated by $p_i-\alpha_i$ with $0\le i\le k-1$. 
This provides a specialisation map 
$$T\to T/\mf m \cong \kk, \hspace{10pt}p_i\mapsto \alpha_i.$$

By right exactness of $\otimes$, 
\[ \frac{\mc Z_{q}\otimes T}{\langle\tilde p(z_1), \dots, \tilde p(z_{|q|})\rangle} \otimes_T T/\mf m \cong \mc Z_{q}/\langle p(z_1), \dots,  p(z_{|q|})\rangle.\]
Thus applying $-\otimes_T T/\mf m$ to the composition
\beq\label{zeromap}
I_k \otimes T \subseteq \mc Z_{q} \otimes T \to \frac{\mc Z_{q}\otimes T}{\langle\tilde p(z_1), \dots, \tilde p(z_{|q|})\rangle}
\eeq
and identifying  $I_k \otimes T \otimes_T T/\mf m$  with $I_k$,
we obtain the natural map
\beq\label{zeromap2}
I_k \to  \mc Z_{q}/\langle p(z_1), \dots,  p(z_{|q|})\rangle.
\eeq
Since by Lemma~\ref{lem:formalfoo} the composition \eqref{zeromap} is zero, the map \eqref{zeromap2} is also zero, as needed. 
\end{proof}  

  \begin{example} \label{eg:chari}
  We thank Vyjayanthi Chari for suggesting the following example to us.
We write elements of $\LslT$ as $yp$ where $y \in \mf{sl}_2$, $p \in \kk[s,s^{-1}]$.
Let $V$ be the adjoint representation of $\mf{sl}_2$, and let $N = V \oplus \CC x$ be the $\mathrm{L} \mf{sl}_2$ representation defined by:  $\CC x $ is the trivial representation, and if $yp \in \LslT$, $v\in V$, then
\[ yp \cdot v = p(1) [y,v] + p'(1) \kappa(y,v) x,\]
where $\kappa$ is the Killing form.
By  construction, $(s-1)^2 \LslT \cdot N = 0$ and so by Proposition~\ref{prop:annihilator2} $\Ker \Phi_2 \subseteq \Ann_{\Ua(\LslT)}(N)$.
However,  $\ker \Phi_1 \cdot N \neq 0$.
To see this, let $e, h, f$ be the standard basis for $\mf sl_2$.
Then  $(e s)h - e( hs) \in \ker \Phi_1$.
 But
 \[ ((es) h - e(hs)) \cdot f = -2e (s-1) \cdot f = -8x \neq 0. \]
\end{example}

\section{Poisson ideals of $\Sa(\Lg)$}\label{POISSON}

In this section, we  prove  Poisson analogues of Theorems~\ref{TmainD+} and \ref{TmainI+}, and then discuss some additional results for $\Sa(\Lg)$.   

We need to modify the notation given in Section~\ref{S:strid} a little. 
Fix a basis $\{b_1, \dots, b_d\}$ of $\mf g$ and use this to induce $\NN^d$-grading on $\Sa(\Lmf g)$ by setting $\deg b_{i[j]}=\epsilon_i$ as in Notation~\ref{not:grade}. 
We denote by $\Sa(\Lmf g)_q$ the graded components with respect to this grading.
We also introduce a similar grading on $$(\Sa(\mf g)[t, t^{-1}])^{\otimes n} \cong \Sa(\mf g^{\oplus n})\otimes \kk[t_1^\pm, \dots, t_n^\pm]$$ by setting $\deg t_i=0, \deg b_{i\{j\}}=\epsilon_i$.  

Next, we define the Poisson analogue $\PhiP$ of $\Phi$:  
\begin{equation}\label{EPPhi}
\PhiP_n: \Sa(\Lmf g) \to \Sa(\mf g^{\oplus n})\otimes \kk[t_1^{\pm 1}, \dots, t_n^{\pm 1}], \hspace{10pt}g s^i \mapsto g_{\{1\}}\otimes t_1^i + \cdots + g_{\{n\}}\otimes t_n^i.\end{equation}
It is clear that this is a $\NN^d$-graded Poisson map. 
Thus
\beq \label{julyQ2+P}
\ker \PhiP_n = \bigoplus_{q \in \NN^d} \ker \PhiP_n \cap \Sa(\Lmf g)_q.
\eeq

\subsection{Poisson ideals and discriminants}

We now give the Poisson analogues of our structural results on ideals of $\Ua(\Lg)$.  These are  relatively straightforward.

\begin{theorem}\label{TPmainD+} 
Fix  a highest weight vector $e \in \mf g$ and let $D_1, D_2,\dots \in \Sa(E)$ be the discriminants defined in Section~\ref{S:disc}. Each nonzero Poisson ideal of $\Sa(\Lmf g)$ contains all but finitely many of the $D_i$.\end{theorem}

\begin{theorem}\label{TPmainI+} Each nonzero Poisson ideal  of $\Sa(\Lmf g)$ contains all but finitely many of the  ideals $\Ker\PhiP_k$.\end{theorem}

We prove these results together.
\begin{proof}[Proof of Theorems~\ref{TPmainD+} and \ref{TPmainI+}]
The scheme of the proof is very similar to the scheme of proof of Theorems~\ref{TmainD+} and \ref{TmainI+}, although   simpler in some details. Let $J$ be a nonzero Poisson ideal of $\Sa(\Lg)$.

Step 1. Let $b_1, \dots, b_d$ be a basis of $\mf g$ so that $b_1=e$ is a highest weight vector.
By Lemma~\ref{lem:JtoH}, $I:= J \cap \Sa(E)$ is a nonzero $H$-ideal of $\Sa(E)$.

Step 2. By Theorem~\ref{Disc for H ideals} we have  $D_k\in I$ for all $k\gg0$.

Step 3. We use the notation of Section~\ref{H IDEALS}. The goal for this step is to prove a Poisson analogue of Proposition~\ref{prop:kernel}.
By analogy with (\ref{Theta}) and~(\ref{Theta2}), for $q\in \NN^d$ we identify $\Sa(\Lmf g)_q$ 
with $\mc Z_q^{\Sg_q}$ through maps
\beq \label{Theta2+P}
  \underline{\Theta}_{\Sg;  \star}^{b^q}: \Sa(\Lmf g)_q
\longleftrightarrow 
\mc Z_q^{\Sg_q}: \underline{\Theta}^{\Sg; \star}_{b^q} \eeq
as in ~\eqref{E:horf}. 
Note that in \eqref{E:horf} for $\Ua(\Lmf g)$ the order of variables/factors in the left-hand side product does matter, while for $\Sa(\Lmf g)$ this order is irrelevant.
Similarly to Proposition~\ref{prop:kernel}, this gives the following.
\begin{proposition}\label{prop:kernel-Poi}
Fix  $q \in \NN^d$ and  $n \in \ZZ_{\geq 1}$.  Then
\[\ker \PhiP_n \cap \Sa(\Lg)_q =  \underline{\Theta}_{b^q}^{\Sg;\star}(\I(\Delta_n^{|q|}) \cap \mc Z_q^{\Sg_q}).\]
\qed
\end{proposition}

Step 4.  We would like to prove the following Poisson analogue of Proposition~\ref{Pkto3k}. 
\begin{proposition}\label{Pkto3k-Poi} Pick $k\ge1$ and recall that $d = \dim \mf g$. 
Then the Poisson ideal  of $\Sa(\Lmf g)$ generated by $\underline{\Theta}_{e^{k+1}}^{\Sg; \star}(\FPol_{k+1}^2)$ contains $\Ker \PhiP_{dk}$.
\end{proposition}
\begin{proof} The proof is very similar to the proof of Proposition~\ref{Pkto3k}.
\end{proof}
Combining Step 2  with Propositions~\ref{prop:kernel-Poi} and~\ref{Pkto3k-Poi} we get the desired result.
\end{proof}

\subsection{$H$-primitive ideals of $\Sa(E)$}\label{SShprim}
From now until the end of Section~\ref{POISSON}, we assume that $\kk$ is uncountable.
Fix a highest weight vector $e \in \mf g$ and adopt Notation~\ref{not:E}, so $E$ has basis $\{ e_i =e \otimes s^i \mid i\in \ZZ\}$.
Let $J$ be an ideal of $\Sa(E)$ and recall the subalgebra $H= \L\mf h = \mf h\otimes\kk[s, s^{-1}]$, which acts on $E$ and on $\Sa(E)$.
Recall  also that  the action of $H$ on $E$ and $\Sa(E)$ reduces to the action of $\Lh:= h\otimes\kk[s, s^{-1}]$
for some $h \in \mathfrak h$ with $[h, e]  = e$.
We denote by $\Core_H(J)$ the sum of all $H$-ideals of $\Sa(E)$ contained in $J$, and we call this ideal {\it the $H$-core of $J$}. 
This is exactly the differential core of $J$ with respect to $H$, which acts on $\Sa(E)$ by derivations.
It is well known that if $J$ is prime then $\Core_H(J)$ is also prime; see for example~\cite[Lemma~3.3.2]{Dix}.
We say that an ideal $I$ is {\it $H$-primitive} if it is the $H$-core of a maximal ideal of $\Sa(E)$. 
An $H$-primitive ideal is automatically prime. 

\begin{definition} For a linear function $\lambda\in E^*$ we set $I_H(\lambda)$ to be $\Core_H(\mf m_\lambda)$ where $\mf m_\lambda$ is the kernel of the evaluation map $\ev_\lambda: \Sa(E)\to\kk, e_i \mapsto \lambda(e_i)$.\end{definition}
By the generalised Nullstellensatz, which holds since $\kk$ is uncountable, all $H$-primitive ideals of $\Sa(E)$ are of the form $I_H(\lambda)$,
and every radical $H$-ideal is an intersection of $H$-primitive ideals.
All these definitions follow the definitions of Poisson ideal, Poisson core, and Poisson primitive ideal,
where we recall, for example, that if $I$ is an ideal of a Poisson algebra $A$ then the {\em Poisson core} of $I$ is the maximal Poisson ideal contained in $I$, and a {\em Poisson primitive} ideal of $A$ is the Poisson core of a maximal ideal of $A$. Compare this with~\cite[Section~1]{Goodearl2006}.

To proceed further we need to introduce a notion of local function on $E$ and more generally on $\Lmf g$. 
To do this we recall that $\Theta^{e}_{\star}: E\to\kk[z_1, z_1^{-1}]=\mc Z_1$ identifies $E = e \otimes \kk[s, s^{-1}]$ with $\mc Z_1$
by sending $e \otimes p (s)$ to $  p(z_1)$. 
Let us set $\mc Z := \mc Z_1 = \kk[z, z^{-1}]$, $\Theta_\star := \Theta^{e}_{\star}$, and $\Theta^\star := \Theta_{e}^{\star}$.
For $x \in \kk^\times$ and $\alpha_0, \dots, \alpha_n \in \kk$, we define a linear function
\begin{equation}\label{Eoneloc}
\chi_{x; \alpha_0, \alpha_1, \dots, \alpha_n}: E\to\kk
\end{equation}
as follows.  Let $g \in E$ and let $p = \Theta_{\star}(g)$.
Then we set 
\begin{equation}
  \chi_{x; \alpha_0, \alpha_1, \dots, \alpha_n}(g) =   
\alpha_0 p(x)+\alpha_1 p^\prime(x)+\dots+\alpha_n p^{(n)}(x).
\end{equation}

\begin{definition}
    Inspired by~\cite[Definition~3.0.2]{PS2} we call a function $\lambda: E\to\kk$  {\em local} if it is a finite sum of  functions of the form~(\ref{Eoneloc}). 
\end{definition}

Even though  $\Lg$ is a very different Lie algebra from the Witt algebra $W = \kk[t,t^{-1}]\del$, we will see that many results from \cite{PS2} about Poisson primitive ideals of $\Sa(W)$ have close analogues for Poisson primitive ideals of $\Sa(\Lg)$ and even for $H$-primitive ideals of $\Sa(E)$.

\begin{lemma}[{cf.~\cite[Theorem~3.1.1]{PS2}}]\label{Lloceq}The following conditions are equivalent for $\lambda\in E^*$:\\
\begin{enumerate}
\item $\lambda$ is a local function;
\item the  two-sided sequence 
\[\dots, \lambda(e_{-2}), \lambda(e_{-1}), \lambda(e_{0}), \lambda(e_{1}), \lambda(e_{2}),\dots\]
is linearly recursive;
\item there exists $\tilde h\in H$ with $[\tilde h, E]\ne(0)$ such that $\lambda([\tilde h, E])=0$;
\item there exists $\tilde e\in E$ with $\tilde e\ne0$ such that $\lambda([\tilde e, H])=0$;
\item there exists a nonzero ideal $I$ of $\mc Z$ such that $\lambda(\Theta^\star(I))=0$.  
\end{enumerate}
\end{lemma}
\begin{proof}
    That $(1)$, $(2)$, $(5)$ are mutually equivalent is a well-known folklore result about two-sided sequences, which was also used in the proof of~\cite[Theorem~3.1.1]{PS2}. 
    Also (3) is clearly equivalent to (4).

    To see that $(3)$ and $(5)$ are equivalent, write 
     $\widetilde h=\sum\limits_{j=-k}^k\beta_j h_j$, where $h_j = h\otimes s^j$.
     Then  
     \[\lambda([\widetilde h, e_i])= \lambda(\sum\limits_{j=-k}^k\beta_je_{i+j})
     =\lambda(\Theta^{\star}((\sum\limits_{j=-k}^k\beta_jz^j)z^i)).\]
     Now $(3)$ holds if and only if for some $\widetilde h = \sum \beta_j h_j$, we have
     $\lambda([\widetilde h, e_i]) = 0$ for all $i$, and by the above
     this is if and only if $\lambda \circ \Theta^\star$ vanishes on $(\sum \beta_j z^j) \mc Z$.
     This is clearly equivalent to $(5)$ and thus the proof is finished. 
\end{proof}
It is clear that if ideals $I_1, I_2$ of $\mc Z$  satisfy condition $(5)$ of Lemma~\ref{Lloceq} then so does $I_1+I_2$. Thus for a local function $\lambda$ we can consider the largest ideal $I_\lambda$ satisfying $(5)$, which is the sum of all  such ideals.  
Note that $I_\lambda \normal \mc Z$, whereas $I_H(\lambda) \normal_H \Sa(E)$.

Let $p_\lambda$ be the unique monic polynomial of minimal degree in $\kk[z]$ which generates $I_\lambda$.
The next result shows that $p_\lambda$ determines $I_H(\lambda)$.

\begin{proposition}\label{PHpr}
\begin{enumerate}
\item We have $I_H(\lambda)\ne (0)$ $\iff$ $\lambda$ is local.
\item If $\lambda$ is local then $I_H(\lambda)=\Sa(E)\Theta^\star(I_\lambda) = \Sa(E) \Theta^\star(\mc Z p_\lambda)$;
\item We have $I_H(\lambda)= I_H(\mu)$ $\iff$ $p_\lambda=p_\mu$.
\end{enumerate} 
\end{proposition}
\begin{proof}
    $(1)$ Pick $\lambda\in E^*$. 
    We would like to show first that $I_H(\lambda)\ne0$ implies (4) of Lemma~\ref{Lloceq}. This follows from the next lemma.
    \begin{lemma}\label{lem:8.14}
    Let $A:=\kk[x_1, x_2, \dots]$ be a polynomial ring in finitely or countably many variables and let $\lambda: A\to \kk$ be a homomorphism of algebras.  Let $\mc D$ be a  subspace of the space of derivations of $A$ and assume that there exists a nonzero ideal $I $ of $ A$ such that $I$ is $\mc D$-stable and $\lambda(I)=0$. Then there exists $\tilde x=\alpha_1 x_1+\alpha_2 x_2+\dots\in A\setminus(0)$ with $\lambda(\mc D\tilde x)=0$ for all $D\in\mc D$.  
    \end{lemma}
    \begin{proof}
        Without loss of generality we can assume that $\lambda(x_i)=0$ for all $i$. Next, choose $f\in I\setminus(0)$. Then $f=f_0+f_1+f_2+\cdots$ where each $f_i$ is a homogeneous polynomial of  total degree $i$.
        Then $\lambda(f)=0$ implies $f_0=0$. 
        Set
        $$\ord f:=\min\{i\mid f_i\ne0\}, \hspace{10pt}\ord I:=\min\{\ord f\mid f\in I\setminus (0)\}.$$ 
        Without loss of generality we can assume that $\ord f=\ord I$. 
        Set $i:=\ord f$. 

        Pick $D\in\mc D$. Then we have $D(f)=\sum_j D(x_j)\partial_jf$, where $\del_j = \frac{\del}{\del x_j}$. We have $D(f)\in I$ and hence $\ord D(f)\ge i$. 
        This implies $(D(f))_{i-1}=0$. Since $\lambda(x_j) = 0$, we have
        $$0=(D(f))_{i-1}=\sum_j\lambda(D(x_j))\partial_jf_i.$$
        Recall that $f_i\ne0$; this implies $\partial_j f_i\ne0$ for some $j$. 
        Consider the space $V_f:={\rm span}_{\kk}\{\partial_j f_i\}$ together with a linear function $\alpha: V_f\to \kk$ such that $\alpha(\partial_jf_i)\ne0$. 
        Then $$0=\alpha(\sum_j\lambda(D(x_j))\partial_jf_i)=\sum_j\lambda(D(x_j))\alpha(\partial_jf_i)=\lambda(D(\sum_j\alpha(\partial_jf_i)x_j)).$$
        Thus $\tilde x=\sum_j\alpha(\partial_jf_i)x_j$ is as needed.
    \end{proof}
    Returning to the proof of Proposition~\ref{PHpr}, let $A:=\Sa(E)$, let $\mc D$  be the derivations defined by $H$, let  $I=I_H(\lambda)$ and use $\lambda\in E^*$ to define a homomorphism $\Sa(E)\to \kk$. This choice together with Lemma~\ref{lem:8.14} shows that (4) of Lemma~\ref{Lloceq} holds and hence $\lambda$ is local. 
    
    We are left to show that if $\lambda$ is local then $I_H(\lambda)\ne0$.  
     Assume that $\lambda$ is local.
     Pick $\tilde h\in H$ for $\lambda$ as in $(3)$ of Lemma~\ref{Lloceq}.
     Then $[\tilde h, E]$ is an $H$-stable subspace of $E$ and $[\tilde h, E]\subseteq \mf m_\lambda$. 
    This implies that the $H$-ideal generated by $[\tilde h, E]$ is contained in $\mf m_\lambda$ and hence  $I_H(\lambda)\ne0$. 

    $(2)$ The argument of $(1)$ implies that $I_H(\lambda)\cap E$ is the maximal $H$-stable subspace of $E\cap \mf m_\lambda$. Such a subspace is clearly identified with $\mc Z p_\lambda$ under the $\Theta$-maps.

    Let $J$ be the $H$-ideal of $\Sa(E)$ generated by $\Theta^\star(\mc Z  p_\lambda)$, which is contained in $I_H(\lambda)$.
    It is easy to verify that $$\GK(\Sa(E)/J)=\deg p_\lambda.$$ 
    On the other hand we have $\GK(\Sa(E)/I_H(\lambda))\ge\deg p_\lambda$ by a straightforward analogue of~\cite[Lemma~2.1.1]{PS1}.  
    This implies that we have a surjection
    \beq\label{domain} \Sa(E)/J\to \Sa(E)/I_H(\lambda) \eeq
    between two rings of the same GK-dimension.
    Now, $I_H(\lambda)$  is $H$-primitive by definition, and thus prime. 
    On the other hand, 
    $\Sa(E)/J \cong \Sa(E/\Theta^\star(\mc Z p_\lambda))$; 
this is true because $\Sa(E)$ is freely generated by the elements of a basis of $E$ and  $\Theta^\star(\mc Z p_\lambda)$ provides exactly a family of linear relations on the elements of this basis and thus the quotient is a polynomial ring $\Sa(E/\Theta^\star(\mc Z p_\lambda))$. 
Thus both rings in \eqref{domain} are domains and so \eqref{domain} is an isomorphism. 

    $(3)$ clearly follows from $(2)$.
\end{proof}

\begin{corollary}\label{CcritP} Let $I$ be an $H$-ideal of $\Sa(E)$.
Then $I$ is $H$-primitive if and only if  $I$  is generated as an $H$-ideal of $\Sa(E)$ by $I\cap E$. 
\end{corollary}
\begin{proof}
    Note that $(0)$ is $H$-primitive since $(0) = I_H(\eta)$ for any non-local $\eta \in E^*$ by Proposition~\ref{PHpr}.
    Assume $I \neq (0)$ and $I$ is $H$-primitive then the result follows from Proposition~\ref{PHpr}. Conversely assume $I \neq (0)$ and $I$ is generated by $I\cap E$. Let $p_I$ be the monic polynomial generator of $\Theta_\star(I\cap E)$ of minimal degree. One can find a local function $\lambda\in E^*$ with $p_\lambda=p_I$; for such $\lambda$ we have $I=I_H(\lambda)$.
\end{proof}
\subsection{Poisson primitive ideals of $\Sa(\Lg)$}

In this subsection we consider Poisson primitive ideals of $\Sa(\Lg)$ and of $\Sa(\widehat{\mf g})$, where $\widehat{\mf g}$ is the affine Lie algebra associated to $\mf g$, that is the unique nontrivial central extension of $\Lg$. 
We will see there is a striking analogy with the structure of Poisson primitive ideals of the symmetric algebras of $W$ and of its central extension the Virasoro algebra $\Vir = W \oplus \kk z$, and we first recall what is known in that situation.
Recall that  a function $\chi\in W^*$ is {\em local} if it is a linear combination of functions of the form
\[ f\del \mapsto \alpha_0 f(x) + \dots + \alpha_n f^{(n)}(x)\]
for some $\alpha_0, \dots, \alpha_n \in \kk$ and $x \in \kk^\times$.  (Note the similarity with local functions on $E$.)
A function $\chi \in \Vir^*$ is {\em local} if $\chi(z) =0$ and $\chi$ induces a local function on $W = \Vir/(z)$.
If $\chi \in W^*$, respectively  if $\chi\in \Vir^*$, we define $P(\chi)$ to be the Poisson core of the maximal ideal of $\Sa(W)$, respectively of $\Sa(\Vir)$, corresponding to $\chi$.
Then we have:

\begin{theorem}[{\cite[Theorem~1.1]{PS2}}]\label{thm:WVir}
\begin{enumerate}
\item $\chi \in W^*$ is local if and only if $P(\chi) \neq (0)$.
\item  $\chi \in \Vir^*$ is local if and only if $P(\chi) \neq (z - \chi(z))$.
\end{enumerate}
\end{theorem}

We now return to $\Lg$. 
Pick $0 \neq b\in\mf g$.  
We define local functions on $b\otimes\kk[s, s^{-1}]$ by  analogy  with local functions on $E=e\otimes\kk[t, t^{-1}]$.
 
 \begin{definition}
     A function $\lambda: \Lmf g\to\kk$ is called {\em local} if
     it is local after restriction to $b\otimes\kk[s, s^{-1}]$ for all $b\in\mf g$.
\end{definition}
Let $b_1, \dots, b_d$ be a basis of $\mf g$. It is easy to check that an equivalent basis-dependent definition is as follows: 
\begin{center}{\it a function $\lambda\in\Lmf g^*$ is local if and only if $\lambda|_{b_i\otimes\kk[s, s^{-1}]}$ is local for all $i$.} \end{center}

If $\lambda \in \Lg^*$, let $\ev_\lambda: \Sa(\Lg) \to \kk$ be the associated evaluation map, and let $\mf m_\lambda  = \ker \ev_\lambda$.  
Let 
$P(\lambda)$ be the Poisson core of  the maximal  ideal  $\mf m_\lambda$.

We first show that Poisson primitive ideals of $\Sa(\Lg)$ are analogous to those of $\Sa(W)$.
\begin{proposition}\label{Ppoipr} For $\lambda\in\Lmf g^*$ we have $P(\lambda)\ne (0)$ $\iff$  $\lambda$ is local.\end{proposition}
\begin{proof}
    We will show first that if $P(\lambda)\ne0$ then $\lambda$ is local. 
    Pick a basis $b_1, \dots,b_d$ of $\mf g$ consisting of highest weight vectors as in the proof of Proposition~\ref{Pkto3k}. 
    By Theorem~\ref{thm:BS}(2) we have that $P(\lambda)\cap \Sa(b_i\otimes\kk[s, s^{-1}])\ne(0)$ for all $i$. 
    Note that $P(\lambda)\cap \Sa(b_i\otimes\kk[s, s^{-1}])$ is an $H$-ideal of $ \Sa(b_i\otimes\kk[s, s^{-1}])$ under the obvious definition.
    By Proposition~\ref{PHpr}, $\lambda|_{b_i\otimes\kk[s, s^{-1}]}$ is local for all $i$ and thus $\lambda$ is local as needed.

    We will show now that if $\lambda$ is local then $P(\lambda)\ne (0)$. 
    By definition there exists $p\in\kk[s]\backslash \{0\}$ so that $\lambda|_{p\Lmf g}\equiv0$.
    Note that $p\Lmf g$ is a Lie ideal of $\Lmf g$ and hence the ideal generated by $p\Lmf g$ is Poisson and is contained in $P(\lambda)$. This implies that $P(\lambda)\ne0$.
\end{proof}

Recall that  $\widehat{\mf g} = \Lg \oplus \kk c$ is the unique nontrivial central extension of $\Lg$.
We say a function $\lambda \in (\widehat{\mf g})^*$ is {\em local} if $\lambda(c) = 0$ and the induced function in $\Lg^*$ is local.
Then Proposition~\ref{Ppoipr} generalises to $\widehat{\mf g}$ as follows.
\begin{proposition}\label{Ppoipr2}
If $\lambda \in (\widehat{\mf g})^*$ let $P(\lambda)$ be the Poisson core of the maximal ideal of $\Sa(\widehat{\mf g})$ corresponding to $\lambda$.
Then $\lambda$ is local $\iff$ $P(\lambda) \neq (c-\lambda(c))$.
\end{proposition}
\begin{proof}
    This follows by combining \cite[Theorem~1.2(1)]{BS}, which shows that if $\lambda(c) \neq 0$ then $(c-\lambda(c))$ is a maximal Poisson ideal of $\Sa(\widehat{\mf g})$, and Proposition~\ref{Ppoipr}.
\end{proof}
We note that this proof is significantly easier than the proof of Theorem~\ref{thm:WVir}(2), since for $\widehat{\mf g}$ we may use powerful results on $\Sa(\widehat{\mf g})$. 
By contrast, by \cite[Corollary~1.3]{PS2}, if $\lambda \neq 0$ then $(z-\lambda)$ is a maximal Poisson ideal of $\Sa(\Vir)$, but this result was only proved using Theorem~\ref{thm:WVir}(2).

We now consider how the coadjoint action of the loop group of $\mf g$ relates to Poisson primitive ideals of $\Sa(\Lg)$.

\begin{notation}\label{not:G} 
Let $\lambda, \mu\in\Lmf g^*$ be local functions, and assume that $\lambda(p\Lmf g)=\mu(p\Lmf g)=0$ for a monic polynomial $p\in\kk[s]$ with $p(0)\ne0$; such $p$ always exists --- for example, one can choose $p=p_\lambda p_\mu$. 
Then both $\lambda$ and $\mu$ come from  uniquely defined functions $\lambda^\uparrow, \mu^\uparrow\in (\Lmf g/p\Lmf g)^*$. 

We denote by $G(p)$ the set of $\kk[s,s^{-1}]/(p)$-points of the algebraic adjoint group $G$ of $\frak g$. It is clear that $G(p)$ acts on $\Lmf g/p\Lmf g\cong\mf g\otimes\kk[s,s^{-1}]/(p)$.
The image of $G(p)$ in ${\rm End}(\Lmf g/p\Lmf g)$ coincides with the adjoint group of $\Lmf g/p\Lmf g$. 
\end{notation}
The following proposition reduces the description of Poisson primitive ideals of $\Sa(\Lmf g)$ to the description of orbits of the algebraic groups defined in Notation~\ref{not:G}.
\begin{proposition} \label{Pprimlm}Let $\lambda, \mu, p, \lambda^\uparrow, \mu^\uparrow, G(p)$ be as in Notation~\ref{not:G}. Then the following conditions are equivalent:
\begin{enumerate}
    \item  $P(\lambda)=P(\mu)$;\\
    \item  $P(\lambda^\uparrow)=P(\mu^\uparrow)$, where both ideals are the obvious Poisson primitive ideals of $\Sa(\Lmf g/p\Lmf g)$;\\
    \item  $G(p)\lambda^\uparrow=G(p)\mu^\uparrow$.
\end{enumerate}\end{proposition}
\begin{proof} $(1)$ and $(2)$ are equivalent thanks to the fact that the ideal generated by the Lie ideal $p\Lmf g$ of $\Lg$ is contained in both $P(\lambda)$ and $P(\mu)$. 

$(2)$ is equivalent to $(3)$ because of the standard bijection between Poisson primitive ideals and coadjoint $G(p)$-orbits for the finite-dimensional algebraic Lie algebra $\Lmf g/p\Lmf g$ (see \cite[{\textsection 3.6}, Proposition]{BrownGordon2003}).
\end{proof}

Recall that $G$ is the adjoint group of $\mf g$.  We can rephrase Proposition~\ref{Pprimlm} in terms of the loop group of $G$.
Let $LG $ be the (polynomial) loop group of $G$, that is $LG$ is the $\kk[s,s^{-1}]$-points of the algebraic group $G$.
(If $\kk = \mb{C}$ this group is denoted $L_{\operatorname{pol}}G_{\mb C}$ in \cite[p. 33]{PressleySegal1986}.)
\begin{proposition}
Assume that $\lambda, \mu\in\Lmf g^*$ and $P(\lambda),  P(\mu)\ne (0)$. 
Then $$P(\lambda)=P(\mu)\iff LG \cdot\lambda=LG\cdot\mu.$$ 
\end{proposition}
\begin{proof}
    This is in essence a version of Proposition~\ref{Pprimlm}. 
    Here we use that if $p\in\kk[s]$ and $ G(p)$ are as in Notation~\ref{not:G}, then,  as is well known to experts, there is a  surjective map $LG = G(\kk[s, s^{-1}])\to G(p)$ --- for a reference, combine~\cite[Lemma~8.2]{PaninStavrovaVavilov2015} and~\cite[Proposition~1.6]{Abe1969}. Therefore $LG$ acts on $(\Lg)^*$ and the embedding  $(\Lg/p\Lg)^*\to (\Lg)^*$ is compatible with these structures. 
\end{proof}

\begin{remark}
    There is a bijection between $H$-primitive ideals of $\Sa(E)$ and ideals of $\kk[s, s^{-1}]$. The story becomes more involved when we pass to $\Sa(\Lg)$. There is a bijection between Lie ideals of $\Lg$ and ideals of $\kk[s, s^{-1}]$. 
    Every $\lambda\in\Lg^*$ with $P(\lambda)\ne0$ is local and thus we can attach $p_\lambda\ne0$ to it. This defines a map from (local) functions to (nonzero) Lie ideals of $\Lg$. It is clear that in the setting of Poisson primitive ideals of $\Sa(\Lg)$ this map is very far from being a bijection; nevertheless it provides an interesting invariant of a Poisson primitive ideal.
\end{remark}

\subsection{The Poisson Dixmier-Moeglin equivalence}\label{SSPDME}
In this subsection we show that the Poisson Dixmier-Moeglin equivalence holds for almost all ideals of $\Sa(\Lmf g)$ (``primitivity''$\Leftrightarrow$``rationality''$\Leftrightarrow$``local closedness''). This effectively means that we provide two conditions for a prime Poisson ideal $I$ of $\Sa(\Lmf g)$ which are equivalent to the condition ``$I$ is Poisson primitive''. This equivalence is widely known in the context of finite-dimensional Lie algebras~\cite{Dix}. It also almost holds for some infinite-dimensional Lie algebras~\cite{PS2}.
Our  proof closely follows~\cite{PS2}.

\begin{definition} Let $I$ be a prime Poisson ideal in a Poisson algebra $A$, which is assumed to be  commutative.
Then $I$ is  {\em Poisson locally closed} if it is locally closed in the Zariski topology on $\PSpec(A)$,
the set of prime Poisson ideals of $A$. 
Further, 
 $I$ is  {\em Poisson rational} if the Poisson centre of $Q(A/I)$ is an algebraic extension of $\kk$.  (That is, if  the Poisson centre of $Q(A/I)$ coincides with $\kk$ since we have assumed that $\kk$ is algebraically closed.)
\end{definition}

We say the {\em Poisson Dixmier-Moeglin equivalence (PDME) holds for $A$} if for any prime Poisson ideal $Q$ of $A$, the conditions
\begin{itemize}
\item[(LC)]
 $Q$ is Poisson locally closed,
\item[(P)] $Q$ is Poisson primitive,
\item[(R)] $Q$ is Poisson rational
\end{itemize}
are equivalent. 
If $\mf g$ is a finite-dimensional Lie algebra, then the PDME holds for $\Sa(\mf g)$ by \cite[Theorem~2]{LL}.

It turns out that the PDME fails for $\Sa(\Lmf g)$, but only because of the  ideal $(0)$. A slightly weakened version of it still holds. 
\begin{theorem}\label{thm:PDME}
\begin{enumerate}
\item The Poisson Dixmier-Moeglin equivalence holds for all prime Poisson  ideals of $\Sa(\Lmf g)$ except for $(0)$:  that is, if $I \neq (0)$ is a prime Poisson  ideal of $\Sa(\Lmf g)$ then 
\begin{center}
    $I$ is Poisson locally closed $\Leftrightarrow$ $I$ is Poisson primitive $\Leftrightarrow$ $I$ is Poisson rational.
\end{center}
\item $(0)$ is Poisson primitive and Poisson rational but not Poisson locally closed.
\end{enumerate}
\end{theorem}

The proof of Theorem~\ref{thm:PDME}  closely follows the proof of~\cite[Theorem~5.3.1]{PS2}. 
\begin{proof}[Proof of Theorem~\ref{thm:PDME}]
By \cite[Theorem~6.3]{LSS}, for prime Poisson ideals of $\Sa(\Lg)$, Poisson locally closed implies Poisson primitive and Poisson primitive is equivalent to Poisson rational.  Thus to prove the theorem for $\Sa(\Lmf g)$, it suffices to prove:

$(1)$ Let $I$ be a prime Poisson  ideal of $\Sa(\Lmf g)$ with $I \neq (0)$.   If $I$ is Poisson primitive, then  $I$ is  Poisson locally closed.

$(2)$ The ideal $(0)$ is Poisson primitive but not Poisson locally closed.

For $(2)$, by applying Proposition~\ref{Ppoipr} to any non-local $\chi \in \Lg^*$ we see that $(0)$ is Poisson primitive.  
To show that $(0)$ is not locally closed, one uses 
 a  mild modification of the proof of~\cite[Theorem~5.3.1(b)]{PS2}.
 We leave the details to the reader.

We now prove $(1)$. Suppose $I$ is Poisson primitive.
By Proposition~\ref{Ppoipr} there exists a local function $\nu\in\Lmf g^*$ such that $I=P(\nu)$. 
By the discussion  after the proof of \cite[Theorem~5.3.1]{PS2}, for the proof of~\cite[Theorem~5.3.1]{PS2} to work
it is sufficient that $\Sa(\Lg)/P(\nu) $ has finite Gelfand-Kirillov dimension. 
Thanks to Theorem~\ref{thm:BS}(2) this holds and  $(1)$ follows. 
\end{proof}
Note that this line of proof can work for a fairly wide class of Lie algebras.

\begin{remark}\label{rem:Prat}
A direct proof that $(0)$ is a Poisson rational ideal of $\Sa(\Lg)$ is not difficult, and may be of interest.
If $(0)$ is not rational, let $0 \neq f, g \in \Sa(\Lg)$ so that $\frac{f}{g}$ is Poisson central in $Q(\Sa(\Lg))$.
Since $\Sa(\Lg)$ is a unique factorisation domain, we can assume that $f,g $ have no common factors.
For $r \in \Sa(\Lg)$ we have
\[ 
0 = \{ r, \frac{f}{g}\} = \frac{\{r, f\} g - f\{r, g\} }{g^2},
\]
so 
$\{r, f\} g = f\{r, g\} $ and $f$ divides $\{r, f\} g$.
Because $f,g$ are coprime, we see that $f$ divides $\{r, f\}$.
In other words, $f \Sa(\Lg)$ is a Poisson ideal of $\Sa(\Lg)$.
This contradicts Theorem~\ref{thm:BS}(2), since clearly 
\[\GKdim \Sa(\Lg)/f \Sa(\Lg) = \infty.\]
\end{remark}

\section{Examples and computations for $\Theta$-maps}\label{EXAMPLES}

The $\Theta$-maps
\[
\Theta_{\star}: \{ \mbox{ $H$-ideals of $\Sa(E)$ }\} 
\leftrightarrows
 \{ \mbox{ $\Sg_\infty$-ideals of $\mc Z_\infty$ }\} :\Theta^\star
\]
defined in Section~\ref{H IDEALS} are some of the central tools of this paper. 
In this section we provide several examples to explore the $\Theta$-maps.  (We have left this section until now since we need some results and notation from Section~\ref{POISSON}.) 
 
 If the $\Theta$-maps were bijective, many of the results of Section~\ref{S:disc} would follow directly from Cohen's fundamental Theorem~\ref{thm:Cohen}.  
 Thus we are particularly interested in examples showing that neither $\Theta^\star$ nor $\Theta_\star$ is injective.
  Example~\ref{Ex1} gives an infinite family of $H$-ideals of $\Sa(E)$ mapped by $\Theta_*$ and $\lceil\Theta\rceil_\star$ to a single ideal of $\mc Z_\infty$,  thus showing that $\Theta_\star$ and $\lceil\Theta\rceil_\star$ are far from being injective.
Then Example~\ref{Exmz} gives two distinct $\Sg_\infty$-ideals of $\mc Z_\infty$ for which $\Theta^\star$ is the same, showing  that $\Theta^\star$ is not injective.
Example~\ref{Exrec} gives two distinct $H$-primitive ideals of  $\Sa(E)$ whose images under $\Theta_\star$ have the same radical, thus  showing that $\sqrt{\Theta_\star(-)}$ is not injective on $H$-primitive ideals.
Finally, in Example~\ref{Exm1} we define two distinct radical ideals $A_\star$ and $C_\star$ of $\mc Z_\infty$ for which $\Theta^\star(A_\star)$ and $\Theta^\star(C_\star)$ are distinct whereas $\sqrt{\Theta^\star(A_\star)}=\sqrt{\Theta^\star(C_\star)}$.
In fact, $\Theta^\star(A_\star)$ is radical and (thus) $\Theta^\star(C_\star)$ is not, so $\Theta^\star$ does not take radical ideals to radical ideals.

Before giving the examples, we restate the definitions of the $\Theta$-maps from Section~\ref{H IDEALS} for the reader's convenience. 
If $I_\star$ is an $\Sg_\infty$-ideal of $\mc Z_\infty$ and  $I^\star$ is an $H$-ideal of $\Sa(E)$, we  define:
\[\Theta^\star(I_\star):= \bigoplus_{k \geq 0} \Theta^{\Sg ; \star}_{e^k}(\lyr_k I_\star ).\]
We define $\Theta_\star(I^\star)$ to be the $\Sg_\infty$-ideal of $\mc Z_\infty$ generated by all $\Theta^{e^k}_{\Sg ; \star}(\gr_k I^\star)$, and we  
 define 
\[ \lceil \Theta \rceil_\star(I^\star) := \bigcup_{k \geq 0} \lceil \Theta^{e^k}_{\Sg ; \star}(\gr_k I^\star ) \rceil. \]

\begin{example}\label{Ex1}
Let $A_\star \normal \mc Z_\infty$ be the kernel of the homomorphism
$$\alpha: \mc Z_\infty \to\kk,\hspace{10pt}z_i\mapsto 1.$$
It is clear that $A_\star$ is an $\Sg_\infty$-ideal.
As an ideal of $\mc Z_\infty$, we have that $A_\star$ is generated by $\{ z_i -1 \mid i \geq 1\}$ and as an $\Sg_\infty$-ideal, it is generated by $z_1-1$.
 Set $A^\star:=\Theta^\star(A_\star)$, and for $m \geq 0$ let $A^\star_{\geq m} = A^\star \cap \Sa^{\geq m}(E)$.
 Then $A^\star_{\geq m_1}\ne A^\star_{\geq m_2}$ for $m_1\ne m_2$.
 On the other hand,  we claim that $A_\star=\Theta_*(A^\star_{\geq m})$ for all $m\ge0$. 
 This is because $z_1+\cdots +z_k -k \in A_\star \cap \mc Z_k^{\Sg_k} \subseteq \Theta_\star(A^\star_{\geq m})$ for all $m \leq k$.
 Thus for any $m$ we have $z_N-1 \in \Theta_\star(A^\star_{\geq m})$ for all $N > m$ and so $\Theta_\star(A^\star_{\geq m}) = A_\star$.

 It is instructive to consider some other properties of $A^\star$.
 Since $A_\star$ is generated as an $\Sg_\infty$-ideal by $z_1-1$,
 $A^\star$ is generated as an $H$-ideal by $\Theta^{\Sg; \star}_e(z_1-1) = e_1-e_0$.
 Therefore $A^\star$ is $H$-primitive, and thus radical, by Corollary~\ref{CcritP}. 
 In fact $A^\star$ is the kernel of the homomorphism
 $$\Sa(E) \to\kk[e],\hspace{10pt}e_i\mapsto e$$
 and is prime.

 Note that $\dim_\kk \Sa(E)/A^\star = \infty$ and $\dim_\kk \mc Z_\infty/A_\star = 1$.
\end{example}
\begin{example} \label{Exmz}Let $B_\star \normal \mc Z_\infty$ be the kernel of the homomorphism 
$$\beta: \mc Z_\infty\to\kk[z, z^{-1}],\hspace{10pt}z_i\mapsto z.$$
It is clear that $B_\star$ is generated by $\{ z_i-z_j \mid i,j\geq 1\}$ and that $B_\star \cap \mc Z_k^{\Sg_k}$ is generated by the  elements \begin{equation}\label{Efr}\frac{z_1^n+z_2^n+\dots+z_k^n}k-(\frac{z_1+z_2+\dots+z_k}k)^n,\hspace{10pt}n\in \ZZ_{\ge0}.\end{equation}


Consider the commutative algebra
$$R =\kk [z,z^{-1}, \delta_1, \delta_2, \dots]/\langle \delta_i\delta_j, \delta_i^2\rangle,$$
where we think of the  $\delta_i$ as  ``infinitely small'' variables of order 1.
Let $\beta': \mc Z_\infty \to R$ be the homomorphism given by mapping
 $z_i \mapsto z+\delta_i, \frac{1}{z_i}\to\frac{1}{z}-\frac{\delta_i}{z^2}$, and let $B'_\star = \ker \beta'$. 
 Since $\beta$ factors through $\beta'$ and $\beta'(z_i-z_j) \neq 0$ if $i \neq j$, we have $B'_\star \subsetneqq B_\star$.
 On the other hand, all elements of the form~\eqref{Efr} 
 belong to $B'_\star$.
 Thus $\Theta^\star(B'_\star) = \Theta^\star(B_\star)$. 
 To summarise, $B_\star$ and $B_\star'$ are distinct $\Sg_\infty$-ideals of $\mc Z_\infty$ such that $\Theta^\star(B_\star)=\Theta^\star(B'_\star).$ 

\end{example}

\begin{example}\label{Exrec} 
We use the notation of Example~\ref{Ex1}. Let  $\lambda: E\to\kk$ be the linear function defined by $\lambda(e_k)=k$. 
It is clear that 
\[\lambda(e_{k+2})-2\lambda(e_{k+1})+\lambda(e_k)=\lambda(e_{k+2}-2e_{k+1}+e_k)=0 \quad \mbox{for all $k$}\]
and it is well known that this is a  linear homogeneous recursive relation of minimal degree for the sequence $\lambda(e_i), i\in\mathbb Z$. 
This implies that $p_\lambda=(z_1-1)^2$ and $ I_H(\lambda)=\Sa(E)\Theta^{\Sg;\star}_{e}(\mc Z_1(z_1-1)^2)$.
We have $(z_1-1)^2\in\Theta_\star(I_H(\lambda))$ and hence $z_1-1\in\sqrt{\Theta_\star(I_H(\lambda))}$. 
Thus $A_\star\subseteq\sqrt{\Theta_\star(I_H(\lambda))}$ by Example~\ref{Ex1}. 
Since $A_\star$ is a maximal ideal of $\mc Z_\infty$,  hence $\sqrt{\Theta_\star(I_H(\lambda))}=A_\star$.

To summarise, $A^\star$ and $I_H(\lambda)$ are distinct $H$-primitive  ideals of $\Sa(E)$ such that $$\sqrt{\Theta_\star(A^\star)}=\sqrt{\Theta_\star(I_H(\lambda))}.$$ 
\end{example}

\begin{example}\label{Exm1} We continue to use the notation of Example~\ref{Ex1}.
Consider the homomorphism
\[\gamma: \mc Z_{\infty}\to\kk, \hspace{10pt}z_1\mapsto -1, \quad z_i\mapsto 1 \mbox{ if $i \geq 2$.}\]
Let 
\[C_\star = \bigcap_{\sigma \in \Sg_\infty} \sigma\cdot(\ker\gamma)\]
be the maximal $\Sg_\infty$-ideal contained in $\ker\gamma$. 
Since $\ker \gamma$ is a maximal ideal of $\mc Z_\infty$, thus
 $C_\star$ is radical. 

For all $k$, by construction,  $\mc Z_k\cap C_\star$ is the intersection of $k+1$ distinct maximal ideals corresponding to the $k+1$ points $(\underbrace{1, 1, \dots, 1}_{k~\mathrm{times}})$ and the same point with a single $1$ replaced by $-1$.
In particular, $\mc Z_k \cap C_\star \subsetneqq \mc Z_k \cap A_\star$ for all $k$ and thus $C_\star \subsetneqq A_\star$.

We have $\mc Z_1\cap C_\star=\mc Z_1(z_1^2-1) \subsetneqq \mc Z_1 \cap A_\star = \mc Z_1(z_1-1)$. 
This implies that  $A^\star = \Theta^\star(A_\star)\supsetneqq\Theta^\star(C_\star)$. 

Recall that $e_1 - e_0 \in A^\star$.  
Since $(z_1-1)(z_2-1) \in C_\star \cap \mc Z_2^{\Sg_2}$, we have
\[ \Theta^{\Sg; \star}_{e^2}((z_1-1)(z_2-1) )= (e_1-e_0)^2 \in \Theta^\star(C_\star).\]

We claim that $A^\star = \sqrt{\Theta^\star(C_\star)}.$
By the previous paragraph and the fact that $A^\star$ is generated as an  $H$-ideal by $e_1-e_0$, we have $A^\star \subseteq \sqrt{\Theta^\star(C_\star)}$.
For all $k$, further, $\gr_k(A^\star)$ has codimension 1 in $\Sa^k(E)$.
Thus either $\sqrt{\Theta^\star(C_\star)} = A_\star$ or $\sqrt{\Theta^\star(C_\star)} \supseteq \Sa^{\geq 1}(E)$.
If $e_1\in \sqrt{\Theta^\star(C_\star)}$, then some $e_1^N\in\Theta^\star(C_\star)$ and so  $z_1z_2\dots z_N\in C_\star$ and
$\gamma(z_1 z_2 \dots z_N) = 0$.
This is not true and hence $\sqrt{\Theta^\star(C_\star)}=A_\star$. 

To summarise,   $A_\star$ and $C_\star$ are distinct radical $\Sg_\infty$-ideals of $\mc Z_\infty$ with $\Theta^\star(A_\star) = \sqrt{\Theta^\star(A_\star)}=\sqrt{\Theta^\star(C_\star)}$.
Since $\Theta^\star(A_\star) \neq \Theta^\star(C_\star)$, thus $\Theta^\star(C_\star)$ is not radical.
\end{example}
\bibliographystyle{amsalpha}
\bibliography{biblio}
\end{document}